\documentclass[11pt]{amsart}
\usepackage{amsmath,amsfonts,amssymb,amsthm,amscd,latexsym,euscript,verbatim, bbold}
\usepackage[all]{xy}

\makeatletter
\def\section{\@startsection{section}{1}%
  \z@{1.1\linespacing\@plus\linespacing}{.8\linespacing}%
  {\normalfont\Large\scshape\centering}}
\makeatother

\theoremstyle{plain}

\newtheorem*{thmA}{Theorem A}
\newtheorem*{thmB}{Theorem B}
\newtheorem*{thmC}{Theorem C}
\newtheorem*{thmD}{Theorem D}

\newtheorem*{conj*}{Root Groups Conjecture}
\newtheorem*{thm1.2}{(1.2) Theorem}
\newtheorem*{thm1.3}{(1.3) Theorem}
\newtheorem*{thm1.4}{(1.4) Theorem}
\newtheorem*{prop*}{Proposition}

\newtheorem{prop}{Proposition}[section]

\newtheorem{thm}[prop]{Theorem}
\newtheorem{cor}[prop]{Corollary}
\newtheorem{lemma}[prop]{Lemma}

\theoremstyle{definition}
\newtheorem{Def}[prop]{Definition}
\newtheorem*{Def*}{Definition}

\newtheorem{examples}[prop]{Examples}

\newtheorem{notation}[prop]{Notation}
\newtheorem*{notation*}{Notation}
\newtheorem{remark}[prop]{Remark}

\newtheorem{prob}{Problem}

\newcommand{\cala}{\mathcal{A}}
\newcommand{\calb}{\mathcal{B}}
\newcommand{\calc}{\mathcal{C}}

\newcommand{\calx}{\mathcal{X}}
\newcommand{\caly}{\mathcal{Y}}

\newcommand{\ff}{\mathbb{F}}

\newcommand{\qq}{\mathbb{Q}}
\newcommand{\rr}{\mathbb{R}}

\newcommand{\zz}{\mathbb{Z}}

\newcommand{\ga}{\alpha}
\newcommand{\gb}{\beta}
\newcommand{\gc}{\gamma}

\newcommand{\gd}{\delta}
\newcommand{\gD}{\Delta}
\newcommand{\gre}{\epsilon}
\newcommand{\gl}{\lambda}

\newcommand{\gvp}{\varphi}
\newcommand{\gr}{\rho}
\newcommand{\gs}{\sigma}

\newcommand{\gt}{\tau}

\newcommand{\charc}{{\rm char}}

\newcommand{\sminus}{\smallsetminus}
\newcommand{\lan}{\langle}
\newcommand{\ran}{\rangle}

\newcommand{\half}{\frac{1}{2}}

\newcommand{\e}{\mathbb{1}}

\numberwithin{equation}{section}

\newcommand\Cl{\mathrm{Cl}}
\newcommand{\Comment}[1]{}
\newcommand{\oZ}{Z}
\newcommand{\oC}{\operatorname{C}}
\newcommand{\Weyl}{\operatorname{W}}
\newcommand{\Symm}[1]{S_{#1}}

\begin{document}
\title[Axial algebras of Jordan type half]{Miyamoto involutions in axial algebras of Jordan type half}
\author[J.~I.~Hall,  Y.~Segev, S.~Shpectorov]{J.~I.~Hall\quad  Y. Segev\quad S.~Shpectorov }
\address{Jonathan, I.~Hall\\
        Department of Mathematics\\
				Michigan State University\\
				Wells Hall, 619 Red Cedar Road, East Lansing, MI 48840\\
				United States}
\email{jhall@math.msu.edu}

\address{Yoav Segev \\
         Department of Mathematics \\
         Ben-Gurion University \\
         Beer-Sheva 84105 \\
         Israel}
\email{yoavs@math.bgu.ac.il}
\address{Sergey Shpectorov\\
         School of Mathematics\\
				 University of Birmingham\\
				Watson Building, Edgbaston, Birmingham, B15 2TT\\
				United Kingdom}
\email{s.shpectorov@bham.ac.uk}


\begin{abstract}
Nonassociative commutative algebras $A$ generated by idempotents
$e$ whose adjoint operators $\operatorname{ad}_e\colon A \rightarrow A$,
given by $x \mapsto xe$, are diagonalizable and
have few eigenvalues are of recent interest. When certain
fusion (multiplication) rules
between the associated eigenspaces are imposed, the structure of
these algebras remains rich yet rather rigid. For example
vertex operator algebras give rise to such algebras.
The connection between the Monster algebra and
Monster group extends to many axial algebras which then
have interesting groups of automorphisms.

Axial algebras of Jordan type $\eta$ are commutative
algebras generated by idempotents whose adjoint operators
have a  minimal polynomial dividing $(x-1)x(x-\eta)$, where
$\eta \notin \{0,1\}$ is fixed, with well-defined and restrictive fusion rules.
The case of $\eta \neq \half$ was thoroughly analyzed by Hall, Rehren, and
Shpectorov in a recent paper, in which axial algebras were introduced.
Here we focus on the case where $\eta=\half$, which is much less understood
and is of a different nature.
\end{abstract}

\date{October 5, 2016}
\maketitle
 
\section{Introduction}

Axial algebras, introduced in \cite{HRS1,HRS},
are certain commutative and nonassociative algebras.
Their definition was motivated by constructions from
the theory of vertex operator algebras.
The historical development is discussed
at length in the introduction to \cite{HRS}.
That paper is focused upon those axial algebras said to have 
Jordan type $\eta$. Specifically, all primitive axial algebras of
Jordan type $\eta \neq  \half$ are essentially classified in  \cite{HRS}.
The purpose of this paper is to understand better the structure of 
primitive axial algebras of Jordan type $\half$. These
remain unclassified and are more varied than those with
$\eta \neq  \half$. Especially, many Jordan algebras
occur, hence the name. 

Throughout this paper $\ff$ is a field of characteristic not $2$.

\subsection{The definition of an axial algebra}\label{sub def axial}\hfill
\medskip

Let $A$ be a commutative algebra over $\ff$ that is not necessarily 
associative. For $a\in A$ and $\lambda \in \ff$ let
\[
A_\lambda(a) := \{x \in A\ |\ xa = \lambda x\}\,.
\]
That is, $A_\lambda(a)$ is the $\lambda$-eigenspace of the adjoint operator

\[
\operatorname{ad}_a \colon  A \rightarrow A,\ x\mapsto xa.
\]
(We allow $A_\lambda(a)=0$.)

An \emph{axis} in $A$ is an idempotent element $a = a^2$ of $A$ such 
that the minimal polynomial of $\operatorname{ad}_a$ is a 
product of distinct linear factors. 
The $\ff$-algebra $A$ is  an \emph{axial
algebra} if it is generated by axes. 
An axis $a \in A$ is \emph{absolutely primitive} if $A_1(a) = \ff a$
(note that this is stronger than being a primitive idempotent).
The structure of axial algebras can be very loose. This is typically
remedied by specifying \emph{fusion rules} (that is, \emph{multiplication
rules}) which restrict how eigenspaces are allowed to multiply. 

Fix $\eta \in \ff$ with $\eta \notin \{0,1\}$.
In this
paper we shall be concerned with axial algebras generated by a
set $\cala$ of  absolutely primitive axes with the following eigenvalue 
and fusion requirements:
\emph{
\begin{enumerate}
\item 
For each $a \in \cala$,
the minimal polynomial of $\operatorname{ad}_a$ divides $(x-1)x(x -\eta)$. 
\item
We have
\[
A_\delta (a)A_\epsilon (a) \subseteq A_{\delta\epsilon} (a)\,,
\]
for $\delta, \epsilon \in \{+,-\}$, and
\[
A_0 (a)A_0 (a) \subseteq A_0 (a)\,,
\]
for all $a \in \cala$. Here $\delta\epsilon$ has the obvious meaning, 
and
\[
A_+ (a) = A_1 (a) \oplus A_0(a)
\quad\text{and}\quad 
A_- (a) = A_\eta (a)\,.
\]
\end{enumerate}
}
An absolutely primitive axis $a$ having these two properties will be called an \emph{$\eta$-axis}.
A \emph{primitive axial algebra of Jordan type $\eta$} is 
a commutative algebra generated by $\eta$-axes. The terminology
arises from the fact that Jordan algebras generated by absolutely primitive axes are primitive axial
algebras of Jordan type $\half$. That particular choice for $\eta$ will
be of the greatest interest to us.
\smallskip

\subsection{Main results}\hfill
\medskip

The current paper has its origins in the paper \cite{HRS} of
Hall, Rehren, and Shpectorov. However, we 
observed that Theorem 5.4 of that paper is false in two
senses---the actual result is false for primitive axial algebras
of Jordan type $\eta=\half$ and, while the result is true for the
cases $\eta \neq \half$, the proof given there is not sufficient.
This paper began as an effort to resolve these difficulties.

To state our main results, let $A$ be a primitive axial
algebra of Jordan type $\eta$.  We need to recall certain definitions.
The {\it Miyamoto involution} $\tau(a)$ corresponding 
to an $\eta$-axis $a\in A$ is the automorphism of $A$
defined by $\gt(a)\colon x\mapsto x_+ - x_{-},$ where $x=x_+ + x_{-},$
with $x_+\in A_+(a)$ and $x_{-}\in A_{-}(a)$.  It is easy to check that
$\tau(a)$ is an automorphism of $A$ of order at most $2$ (see
Definition \ref{def miyamoto} below).

\subsubsection{Results concerning the structure of $A$}\hfill
\medskip

Consider the (undirected) graph $\gD$
on the set of all  $\eta$-axes of $A,$ where distinct $a, b$ form
an edge if and only if $ab\ne 0$ (see \S\ref{sect delta}). 
In \S\ref{sect delta} we show the existence of a certain decomposition
of our axial algebras.

%
\begin{thmA}
Let $A$ be a primitive axial algebra of Jordan type $\eta$.
Let $\{\gD_i\mid i\in I\}$
be the connected components of $\gD$ and let $A_i$ be
the subalgebra of $A$ generated by the axes in $\gD_i$.  Then
\begin{enumerate}
\item
$A=\sum_{i\in I}A_i$ is the sum of its ideals $A_i;$ 
\item
$A_iA_j=0,$ for distinct $i, j\in I;$

\item
for each $i\in I$ exactly one of the following holds:
\begin{enumerate}
\item
the map $a\mapsto\gt(a),\ a\in\gD_i$ is injective.

\item
$A_i$ is a Jordan algebra of Clifford type.
\end{enumerate}
\end{enumerate} 
\end{thmA}

Jordan algebras of Clifford type are discussed in \S\ref{sect jordan}.
Theorem \ref{thm a} below give a more detailed (and refined) version
of Theorem A.

To prove part (3b) of Theorem A we prove:  

%
\begin{thmB}
Let $A$ be a primitive axial algebra of Jordan type $\eta$.
Assume that $\gD$ is connected and that there are two distinct
$\eta$-axes $a, b\in A$ such that $\gt(a)=\gt(b)$.
Then $\eta=\half,\ a+b=\e$ is the identity of $A$ and $A$ is a Jordan
algebra of Clifford type.
\end{thmB}

The proof of Theorem B uses Theorem \ref{thm jvb} and Proposition \ref{prop e}.

\subsubsection{Results concerning $3$-transpositions}\label{sub intro 3-trans}\hfill
\medskip

Let $G$ be a group generated by a normal set of involutions $D$.
Recall that $D$ is called {\it a set of $3$-transpositions in $G$}
if $|st|\in\{1,2,3\},$ for all $s,t\in D$.  The group
$G$ is then called a {\it $3$-transposition group}.

Let $\cala$ be a generating set
of $\eta$-axes in $A$.  Suppose that $a^{\gt(b)}\in\cala,$ for all $a, b\in\cala,$
where $a^{\gt(b)}$ is the image of $a$ under the Miyamoto
involution $\gt(b)$ (and this and similar notation will
prevail throughout this paper).  In other words, assume
that $\cala$ is {\it closed}.  

As we will see,
the set $D:=\{\gt(a)\mid a\in\cala\text{ and }\gt(a)\ne{\rm id}\}$ is a normal
set of involutions in $G=\lan D\ran$.  Suppose $D$ is a set
of $3$-transpositions in $G,$
then we call $A$ {\it a $3$-transposition algebra with respect to $\cala$}.

In \cite{HRS} it is shown that {\it every} primitive axial 
algebra of Jordan type $\eta\ne\half$ is a $3$-transposition
algebra with respect to {\it any} closed generating set of $\eta$-axes
(see Theorem \ref{thm hrs}).
The case $\eta=\half$ is very different.  
However, the following theorem holds in the case where
$A$ is a Jordan algebra of Clifford type (see \S\ref{sect 3trans}).

%
\begin{thmC}
Assume $\operatorname{char}(\ff) \neq 3$ and that
$A$ is a Jordan algebra of Clifford type that is
additionally a $3$-transposition algebra with
respect to the closed set $\cala$ of $\half$-axes.
Assume further that 
$D:=\{\gt(a)\mid a\in\cala\text{ and }\gt(a)\ne{\rm id}\}$ is
a conjugacy class in $G=\lan D\ran$.  Then $G$ is of symplectic
type and no subgroup
$H=\langle D \cap H\rangle$ is
isomorphic to a central quotient of $\Weyl_2(\tilde{D}_4)$. 
\end{thmC}

For more detail about the terminology in Theorem C, see \S\ref{sect 3transpositions}.
In that section we prove a result
about $3$-transposition groups which is of independent
interest.
 
%
\begin{thmD}
Let $G$ be a finite $3$-trans\-pos\-i\-tion group of symplectic type
generated by the conjugacy class $D$ of $3$-trans\-pos\-i\-tions,
such that there is no subgroup
$H=\langle D \cap H\rangle$ 
isomorphic to a central quotient of $\Weyl_2(\tilde{D}_4)$. 

Then there is an $n \in \zz^+$
with $G$ a central quotient of 
$\Weyl(A_n)$ for $n \ge 2$, $\Weyl(D_n)$ for $n\ge 4$,
or $\Weyl(E_n)$ for $n \in \{6,7,8\}$.
\end{thmD}

Of course Theorem $D$ can be used in conjunction with Theorem C
to impose severe restrictions on the groups $G$ that can
occur in Theorem C.

%
\subsection{Some open problems}\hfill
\medskip

Here we list certain remaining open problems.

\begin{prob}
Let $A$ be a primitive axial algebra of Jordan type $\half$
and assume that the graph $\gD$ above is connected and that
the map $a\mapsto\tau(a)$ is bijective on the set of $\half$-axes.
\begin{enumerate}
\item[(i)]
What is the structure of subalgebras of $A$ generated by three $\half$-axes?
 
\item[(ii)]
Suppose $A$ is generated by a finite set of $\half$-axes.  Is $A$
finite dimensional over $\ff$?

\item[(iii)]
What else can be said about the structure of $A$?
\end{enumerate}
\end{prob}

\begin{prob}
Let $A$ be a Jordan algebra of Clifford type
and assume that the graph $\gD$ above is connected.
Let $\cala$ be a closed subset $($see above$)$
of  generating $\half$-axes in $A$ such that the map $a\mapsto\tau(a)$
is injective on $\cala$.  Classify these sets $\cala$.
\end{prob}

\begin{prob}
Are there restriction on $G$ and $D$ if we assume in Theorem C
that $\charc(\ff)=3$?
\end{prob}
\tableofcontents
 
\section{Notation and some definitions}
In this section we assemble the notation and definitions that will prevail
throughout this paper.  Other, more specific ones, will be given
in the beginning of each of the following sections.

As mentioned above, throughout this paper $\ff$ is a field of {\it characteristic not $2$}.
Also, $A$ is a primitive axial algebra of Jordan type $\eta$.

\begin{Def}[Miyamoto involution]\label{def miyamoto}
Let $B$ be an algebra over the field $\ff$ (not necessarily associative).
\begin{enumerate}
\item
Suppose that $B$ is a direct sum $B=B_{+}\oplus B_{-}\, ,$ such that $B_+$
and $B_{-}$ are subspaces of $B$ and such that $B_{\gd}B_{\gre}=B_{\gd\gre},$
for $\gd,\gre\in\{+,-\}$.  For $x\in B,$ write $x=x_++x_{-}\, ,$ with $x_{\gre}\in B_{\gre}$.
The {\it Miyamoto involution} corresponding to the above decomposition of $B$
is the map $\gt\colon B\to B$ defined by $x^{\gt}=x_+-x_{-}\, ,$ for all $x\in B$. It is easy to check that
$\gt$ is an automorphism of $B$ of order at most $2$. (It has order $2$ if and only if $B_{-}\ne 0$.) 

\item
Let $x\in B,$ and assume that $B$ decomposes into a direct sum $B_+(x)\oplus B_{-}(x)$
of ${\rm ad}_x$-invariant subspaces of the adjoint operator ${\rm ad}_x\colon B\to B,\ b\mapsto bx$.
Suppose further that $B_+(x)$ and $B_{-}(x)$ satisfy the rules as in (1) above.
Then we denote the corresponding Miyamoto involution by $\gt(x)$ (where $B$ is understood
from the context) and call $\tau(x)$ the {\it Miyamoto involution corresponding to $x$}.

\item
When $A$ is a primitive axial algebra of Jordan type $\eta$ and $a\in A$ is an
$\eta$-axis, then $\gt(a)$ will denote the Miyamoto involution corresponding to $a$
(as in (2) above, recall the definition of $A_{+}(a)$ and $A_{-}(a)$ from
subsection \ref{sub def axial}).  It is easy to check that if $\gr$ is an automorphism
of $A,$ then $\gt(a^{\gr})=\gt(a)^{\gr}$.  This fact will be used throughout
this paper without further mention. 
\end{enumerate}
\end{Def}

\begin{notation}[Notation and definitions related to axes]\label{not axes}

\noindent
\begin{enumerate}
\item
We denote by $\calx$ the set of all $\eta$-axes in $A$.

\item
For a subalgebras $N\subseteq A$ and $x\in N,$ we denote by $N_{\gl}(x)$
the $\gl$-eigenspace of the adjoint endomorphism
${\rm ad}_x\colon b\mapsto bx$ of $N$. (Recall that we allow the possibility
$N_{\gl}(x)=0$.)

\item
We let $\calx^1:=\{a\in\calx\mid \gt(a)={\rm id}\}$ and $\calx^{\eta}:=\calx\sminus\calx^1$.

\item
For a subset $\calb\subseteq\calx,$ we let $\calb^1=\calb\cap\calx^1$ and $\calb^{\eta}=\calb\cap\calx^{\eta}$. 

\item
A subset $\calb\subseteq\calx$ is {\it closed} if $\calb^{\tau(b)}\subseteq\calb,$
for all $b\in\calb$.

\item
For a set $\calb\subseteq \calx,$ the {\it closure}
of $\calb$ in $\calx$ is the intersection 
of all closed subsets of $\calx$ containing $\calb$.
We denote it by $[\calb]$. 

\item
For a subset $\calb\subseteq\calx,$ we denote by $D_{\calb}:=\{\tau(b)\mid b\in\calb\}$
and $G_{\calb}=\lan D_{\calb}\ran$.
\end{enumerate}
\end{notation}

\begin{notation}[General notation for algebras and subalgebras]\label{not alg}
Let $a,b\in\calx$ with $a\ne b$.
\begin{enumerate}
\item
For a subset $\calb\subseteq \calx,$  we denote by $N_{\calb}$ the subalgebra
of $A$ generated by $\calb$.  If $\calb=\{a,b\},$ we sometimes
write $N_{\calb}=N_{a,b}$.

\item
Note that $N_{a,b}$ satisfies the multiplication rules of Proposition \ref{prop table}
below.  We use the notation $\gvp_{a,b},\ \pi_{a,b}$ and $\gs_{a,b}$
as in Proposition \ref{prop table}.

\item
We denote by $1_{a,b}$ the identity element
of $N_{a,b}$ if $N_{a,b}$ is $3$-dimensional and has an identity element.

Note that by Theorem \ref{thm betaphi}(3),\, $ N_{a,b}$ contains 
an identity element if and only if  $\gs_{a,b}\ne 0$
and $\pi_{a,b}\ne 0$.  In this case the  identity
element of $N_{a,b}$ is $1_{a,b}=\frac{1}{\pi_{a,b}}\gs_{a,b}$. 
\end{enumerate}
\end{notation}

\begin{notation}[Some specific two generated algebras]\label{not algebras}
The following $2$-generated algebras were defined in \cite{HRS}.
In several cases though we changed notation.
Let $a,b\in\calx$ with $a\ne b$.
\begin{enumerate}
\item 
If $ab=0,$ we denote $N_{a,b}=2B_{a,b},$
thus $N_{a,b}$ is $2$-dimensional.

\item
If $ab=-a-b,$ 
we denote $N_{a,b}=3C(-1)^{\times}_{a,b},$ 
thus $N_{a,b}$ is $2$-dimensional 
(see Lemma \ref{lem 3c(-1)times}).

\item
If $ab=\half a+\half b$ we denote $N_{a,b}=J_{a,b},$
thus $N_{a,b}$ is $2$-dimensional
(see Lemma \ref{lem jab}). Our notation here for this
algebra differs from the notation in \cite{HRS}.

\item
If $N_{a,b}$ is $3$-dimensional, we denote $N_{a,b}=B(\eta,\gvp_{a,b})_{a,b},$
where $\gvp_{a,b}\in\ff$ is defined in Proposition \ref{prop table}
(see Theorem \ref{thm betaphi}).

\item
We denote $3C(\eta)_{a,b}=B(\eta,\half\eta)_{a,b}$  
(see Lemma \ref{lem 3ceta} and Remark \ref{rem 3ceta}).

\end{enumerate}
\end{notation}

\section{Preliminaries}
\numberwithin{prop}{subsection}
In this section we give some preliminary properties
of the various algebras from Notation \ref{not algebras}.  In addition,
we assemble some preliminary results.

\subsection{Details about the algebras in Notation \ref{not algebras}}\label{sub details}

\begin{prop}[Proposition 4.6 \cite{HRS}]\label{prop table}
Let $a,b\in\calx$ with $a\ne b$.  Let $\gs=\gs_{a,b}=ab-\eta a-\eta b\in N_{a,b}$. Then there exists 
a scalar $\gvp=\gvp_{a,b}\in\ff$ such that if we set $\pi=\pi_{a,b}=(1-\eta)\gvp-\eta,$
then
\begin{enumerate}
\item
$ab=\gs+\eta a+\eta b;$

\item
$\gs v=\pi v,$ for all $v\in\{a,b,\gs\}.$
\end{enumerate}
\end{prop}

\begin{lemma}\label{lem 2dim}
Let $a,b$ be two distinct $\eta$-axes in $A$ and suppose that
$N_{a,b}$  is $2$-dimensional.
Then  
\begin{enumerate}
\item
one of the following three statements holds:
\begin{enumerate}
\item
\begin{enumerate}
\item
$\gvp_{a,b}=0,\pi_{a,b}=-\eta$ and $\gs_{a,b}=-\eta a-\eta b,$ also

\item
$ab=0$ and $N_{a,b}=2B_{a,b}$.
\end{enumerate}
\item
\begin{enumerate}
\item
$\eta=-1,\ \gvp_{a,b}=-\half,\ \pi_{a,b}=0$ and $\gs_{a,b}=0,$ also

\item
$ab=-a-b$ and $N_{a,b}=3C(-1)^{\times}_{a,b}$.
\end{enumerate}
\item
\begin{enumerate}
\item
$\eta=\half,\ \gvp_{a,b}=1, \pi_{a,b}=0$ and $\gs_{a,b}=0,$ also 

\item
$ab=\half a+\half b,$ and $N_{a,b}=J_{a,b}$.
\end{enumerate}
\end{enumerate}
\item
If $ab\ne 0,$ then $N_{a,b}$ does not have an identity element.
\end{enumerate}
\end{lemma}
\begin{proof}
Suppose first that $\gs_{a,b}\ne 0$.  Let $\gs=\gs_{a,b}$ and $\pi=\pi_{a,b}$.
Write $\gs =\ga a+\gb b$.  Multiplying by $a$ we get that $\pi a=\ga a+\gb ab$.  Suppose
$\gb=0$.  Then $\pi a=\ga a,$ thus $\pi\ne 0,$ and  $\ga=\pi$.  Multiplying by $b$ we get that $\pi b=\pi ab,$
so $ab=b$.  But this contradicts the absolute primitivity of $a$.  Hence $\gb\ne 0$.  Similarly, $\ga\ne 0$.
   
Thus we have $ab=\frac{\pi-\ga}{\gb}a$ and similarly multiplying by $b$ we get that $ab=\frac{\pi-\gb}{\ga}b$.
Hence we must have $\ga=\gb=\pi\ne 0$.  It follows that  
$ab=0$ and $N_{a,b}=2B_{a,b}$.  Further $\gs=-\eta a-\eta b$.
Now $\pi a=\gs a=-\eta a$.  Hence $\pi=-\eta$ and then $\gvp_{a,b}=0$.
This shows part (1a).

Suppose now that $\gs=0$.   
Then $ab=\eta a+\eta b$.  Thus, if $a(\ga a+\gb b)=0,$ for some vector $\ga a+\gb b\in N_{a,b},$
then $\ga a+\gb (\eta a+\eta b)=0$.  Hence $\gb\eta =0,$ so $\gb=0,$
and then also $\ga =0,$ and we see that the $0$-eigenspace of $a$ is $\{0\}$
and so is the $0$-eigenspace of $b$.
 
We now compute the $\eta$-eigenspace of $a$.  
Clearly, it is a $1$-dimensional space
spanned by some $\ga a+\gb b,$ with $\ga\ne 0\ne \gb$.

Now since $A$ is of Jordan type $\eta,$ we must have $(\ga a+\gb b)^2\in \ff a$.
Thus
\[
\ga^2a+\gb^2b+2\ga\gb(\eta a+\eta b)\in\ff a.
\]
It follows that $\gb^2+2\ga\gb\eta=0,$ or
\[
\gb=-2\ga\eta.
\]
Canceling $\ga,$ we may assume that the $\eta$-eigenspace of $a$ is spanned
by $a-2\eta b$.

Next we have $a(a-2\eta b)=\eta(a-2\eta b)$.  Hence 
\[
a-2\eta (\eta a+\eta b)=\eta a-2\eta^2 b.
\]
It follows that $2\eta^2+\eta-1=0$ so, $\eta =-1$ or $\eta =\half$.
If $\eta=-1,$ then $ab=-a-b,$ and this is the definition of $3C(-1)^{\times}_{a,b}$.
This shows part (1bii).
Since $0=\gs a=\pi a,$ it follows that $\pi=0$ and then $\gvp=-\half$
and (1bi) holds.

If $\eta=\half,$ then $ab=\half a+\half b$.
This is the algebra $J_{a,b}$.  This shows part (1cii).
As above,  $\pi=0$ and then $\gvp=1$. 

Part (2) is an easy calculation and we omit the details.
\end{proof}

\begin{thm}\label{thm betaphi}
Let $a, b$ be two distinct $\eta$-axes in $A$ and
assume that $N:=N_{a,b}$  is $3$-dimensional.
Set $\gs=\gs_{a,b},\ \gvp=\gvp_{a,b}$ and $\pi=\pi_{a,b}$. Then 
\begin{enumerate}
\item
$\gs\ne 0;$

\item
$\gs z=\pi z,$ for $z\in\{a,b,\gs\};$  

\item
$\pi\ne 0$ if and only if $N$ contains an identity element $1_{a,b}=\frac{1}{\pi}\gs;$

\item
$ab=\gs+\eta a+\eta b;$ in particular, if $\pi\ne 0$ then $ab=\pi 1_{a,b}+\eta a+\eta b;$

\item
$N_1(x)=\ff x,$ $N_0(x)=\ff(\pi x-\gs)$ and $N_{\eta}(x)=\ff((\eta-\gvp)x+\eta y+\gs);$

\item
$x^{\tau(y)}=-\frac{2}{\eta}\gs-\frac{2(\eta-\gvp)}{\eta}y-x$, for $\{x,y\}=\{a,b\};$
 
\item
$xx^{\tau(y)}=-\frac{2(\eta-\gvp)}{\eta}\gs-\left(2(\eta-\gvp)+\frac{2}{\eta}\pi+1\right)x-2(\eta-\gvp)y,$
for $\{x,y\}=\{a,b\}.$
\end{enumerate}
\end{thm}
\begin{proof}
Note that $N_{a,b}$ is isomorphic to the algebra $B(\eta,\gvp)$ defined in
\cite[Theorem 4.7, p.~98]{HRS},
with $(a,b)$ in place of $(c,d)$ and $\gs_{a,b}$ in place of $\rho$.
Hence, parts (1)--(5) are \cite[Theorem 4.7(a), p.~99]{HRS}. 
Note that if $\pi=0$ and $N$ contains an identity element $1_{a,b},$
then $\gs=1_{a,b}\gs=0,$ a contradiction.

For (6) we have
\begin{alignat*}{3}
x&=\textstyle{-\frac{1}{\eta}\gs-\frac{\eta-\gvp}{\eta}y+\frac{1}{\eta}\left((\eta-\gvp)y+\eta x+\gs\right)}& &\iff\\
x^{\tau(y)}&=\textstyle{-\frac{1}{\eta}\gs-\frac{\eta-\gvp}{\eta}y-\frac{1}{\eta}\left((\eta-\gvp)y+\eta x+\gs\right)}& &\iff\\
x^{\tau(y)}&=\textstyle{-\frac{2}{\eta}\gs-\frac{2(\eta-\gvp)}{\eta}y-x.}
\end{alignat*}
Finally, for (7) we have
\begin{align*}\textstyle{
xx^{\tau(y)}}&\textstyle{=x(-\frac{2}{\eta}\gs-\frac{2(\eta-\gvp)}{\eta}y-x)}\\ 
&=\textstyle{-\frac{2}{\eta}\pi x-\frac{2(\eta-\gvp)}{\eta}xy-x}\\
&=\textstyle{-\frac{2}{\eta}\pi x-\frac{2(\eta-\gvp)}{\eta}(\gs+\eta x+\eta y)-x}\\
&=\textstyle{-\frac{2}{\eta}\pi x-\frac{2(\eta-\gvp)}{\eta}\gs-2(\eta-\gvp) x-2(\eta-\gvp) y-x}\\
&=\textstyle{-\frac{2(\eta-\gvp)}{\eta}\gs-\left(2(\eta-\gvp)+\frac{2}{\eta}\pi+1\right)x-2(\eta-\gvp)y .}\qedhere
\end{align*}
\end{proof}

\begin{lemma}\label{lem 3ceta}
Let $a, b$ be two distinct $\eta$-axes in $A$ and suppose that
$N_{a,b}=B(\eta,\,\half\eta)_{a,b}$ (in particular, $N_{a,b}$ is $3$-dimensional).  Set $N=N_{a,b},\ \pi=\pi_{a,b}$
and $\gs=\gs_{a,b}$.   Then $\pi =-\half(\eta+\eta^2),$
and
\begin{enumerate}
\item 
\begin{enumerate}
\item
$ab=\gs+\eta a+\eta b;$

\item
$N_{\eta}(x)=\ff(\gs+\half\eta x+\eta y),$ for $\{x, y\}=\{a,b\};$
\item
$x^{\tau(y)}=-\frac{2}{\eta}\gs-x-y,$
for $\{x,y\}=\{a,b\};$
 
\item
$a^{\tau(b)}=b^{\tau(a)},$
so $|\tau(a)\tau(b)|=3.$
\end{enumerate}

\item
If $\eta\ne -1,$ then $1_{a,b}=\frac{1}{\pi}\gs$ is the identity of $N$ 
while if $\eta=-1,$ then $N$ has no identity element.
\end{enumerate} 
\end{lemma}
\begin{proof}
Part  (1a) follows from Theorem \ref{thm betaphi}(4).
Also
\[\textstyle{
\pi=(1-\eta)\gvp-\eta=(1-\eta)\frac{\eta}{2}-\eta=-\half(\eta+\eta^2).}
\]
Next, by Theorem \ref{thm betaphi}(5), for $x\in\{a,b\},$
$N_{\eta}(x)$ is spanned by $(\eta-\half\eta)x+\eta y+\gs,$
so  (1b) holds.
By \ref{thm betaphi}(6), $x^{\tau(y)}=-\frac{2}{\eta}\gs-\frac{2(\eta-\half\eta)}{\eta}y-x$, for $\{x,y\}=\{a,b\},$
this shows (1c) and (1d).  Part (2) follows from
Theorem \ref{thm betaphi}(3). 
\end{proof}

\begin{remark}\label{rem 3ceta}
Let $a, b$ be two distinct $\eta$-axes in $A$ and suppose that
$N_{a,b}=B(\eta,\,\half\eta)_{a,b}$.  Set $c_0=a, c_1=b$ and
$c_2=-\frac{2}{\eta}\gs_{a,b}-a-b$.
Then it is readily verified that $\{c_0, c_1, c_2\}$ is a basis of $N_{a,b},$ and  for $\{i, j, k\}=\{0, 1, 2\}$
we have  $c_i^2=c_i$ and $c_ic_j=\half\eta(c_i+c_j-c_k)$.  Thus
$N_{a,b}$ is the algebra denoted $3C(\eta)$
in \cite[Example 3.3, p.~90]{HRS}.
See also \cite{HRS}, p.~91, line 9 (with $(a,b)$ in place of $(c_0,c_1)$).
This explains our notation $3C(\eta)_{a,b}$ (see Notation \ref{not algebras}(5)).
\end{remark}

\begin{lemma}\label{lem halfhalf}
Let $a,b\in A$ be two  distinct $\eta$-axes. 
Suppose that $\eta=\gvp_{a,b}=\half$.  Then
$N:=N_{a,b}$ is $3$-dimensional, so $N=B(\half,\half)_{a,b}$ and
\begin{enumerate}
\item
$\pi_{a,b}=-\frac{1}{4};$

\item
$x^{\tau(y)}=1_{a,b}-x,$ for $\{x,y\}=\{a,b\};$

\item
$N_{\half}(x)=\ff(1_{a,b}-2y),$ for $\{x,y\}=\{a,b\};$

\item
$|\tau(a)\tau(b)|\in\{2,4\};$

\item
$|\tau(a)\tau(b)|=2$ if and only if $\tau(x)=\tau(1_{a,b}-x),$ for $x\in\{a,b\};$  

\item
if $|\tau(a)\tau(b)|=4$ then $\tau(a)\tau(1_{a,b}-a)=\tau(b)\tau(1_{a,b}-b)=:t$
and $Z(\lan \tau(a),\tau(b)\ran)=\lan t\ran;$

\item
$N=B(\half,\half)_{x,y},$ for all $x\in\{a, 1_{a,b}-a\}$ and   $y\in\{b, 1_{a,b}-b\}.$
\end{enumerate}
\end{lemma}
\begin{proof}
By Lemma \ref{lem 2dim}, $N$ is $3$-dimensional.  
Let $\gvp=\gvp_{a,b},\ \pi=\pi_{a,b},\ \gs=\gs_{a,b}$ and $\e=1_{a,b}$.
Since $\eta=\half,$ 
we have $\pi=(1-\eta)\gvp-\eta=(1-\half)\half-\half=-\frac{1}{4}$.
By Theorem \ref{thm betaphi}(6),
$x^{\tau(y)}=-4\gs-x=\e-x.$
By Theorem \ref{thm betaphi}(5) part (3) holds.
Since $a^{\tau(a)}=a$ and $b^{\tau(a)}=\e-b,$ it follows that $N=B(\half,\half)_{a,\ (\e-b)}$.
Similarly we see that (7) holds.

From (7) it follows that $x^{\tau(y)\tau(\e-y)}=x$, and clearly $y^{\tau(y)\tau(\e-y)}=y,$
for $\{x,y\}=\{a,b\}$.  Hence $\tau(y)\tau(\e-y)\in Z(H),$ where $H=\lan\tau(a),\tau(b)\ran$.
We have 
\begin{gather*}
\tau(\e-b)\tau(b)=\tau(b^{\tau(a)})\tau(b)=(\tau(a)\tau(b))^2=\tau(a)\tau(a^{\tau(b)})\\
=\tau(a)\tau(\e-a)\in Z(H).
\end{gather*}  
If $\tau(x)=\tau(\e-x)$ for $x=a$ or $b,$
then $|\tau(a)\tau(b)|=2$ and clearly if $|\tau(a)\tau(b)|=2,$ then $\tau(x)=\tau(\e-x)$.
This completes the proof of the lemma.
\end{proof}

\begin{remark}\label{rem halfhalf}
Suppose that $\eta=\half$. Let $a,b\in \calx$.
Note that $N_{a,b}=B(\half,\half)_{a,b}$ if and only if $N_{a,b}$ contains
an identity $1_{a,b}$ and $ab=-\frac{1}{4}1_{a,b}+\half a+\half b$.  
\end{remark}

\begin{lemma}\label{lem 3c(-1)times}
Let $C:=3C(-1)^{\times}_{a,b}$.
Then
\begin{enumerate}
\item
$ab=-a-b,\ \gs_{a,b}=0,\ \pi_{a,b}=0$ and $\gvp_{a,b}=-\half;$
\item
$C$ has no identity element;

\item
$C_{-1}(a)=a+2b$ and $C_{-1}(b)=b+2a;$
\item
$ab=a^{\tau(b)}=b^{\tau(a)}=-a-b,$ so $(\tau(a)\tau(b))^3={\rm id};$

\item
if $\charc(\ff)\ne 3,$ the only non-zero idempotents in $C$ are $a, b, -a-b,$
and they are all  $-1$-axis with $C_{-1}(-a-b)=\ff(a-b)$;

\item
if $\charc(\ff)= 3$ then the non-zero idempotents in $C$ are
\[
\mathcal{E}:=\{\ga a+(1-\ga)b\mid \ga\in\ff\},
\]
and each $e\in\mathcal{E}$ is a $-1$-axis in $C$ with $C_{-1}(e)=\ff(a-b).$

\end{enumerate}
\end{lemma}
\begin{proof}
Part (1) is by definition, and part (2) is
Lemma \ref{lem 2dim}(2).
\medskip

\noindent
(3):\quad 
$a(a+2b)=a+2ab=a-(2a+2b)=-a-2b,$ and similarly for $b$.
\medskip

\noindent
(4):\quad We have $a=-\half b+\half(b+2a),$ so, by definition,
$a^{\tau(b)}=-\half b-\half(b+2a)=-a-b$.
By symmetry, $b^{\tau(a)}=-a-b$.
It follows that $a^{\tau(b)}=b^{\tau(a)}$, and then $|\tau(b)\tau(a)|=3$.
\medskip

\noindent
(5\&6):\quad
Assume that $e:=\ga a+\gb b$ is a non-zero idempotent.
Then
\[
\ga a+\gb b=(\ga a+\gb b)^2=\ga^2a+\gb^2b+2\ga\gb ab=(\ga^2-2\ga\gb)a+(\gb^2-2\ga\gb)b.
\]
If $\ga=0,$ then $\gb=1,$ and $e=b$.  Similarly if $\gb =0,$ then $\ga=1$ and $e=a$.
Suppose $\ga\ne 0\ne\gb$. 
Then $\ga^2-2\ga\gb=\ga,$ and so $\ga-2\gb=1$.  Similarly $\gb-2\ga=1,$
and we see that if $\charc(\ff)\ne 3,$ then $\ga=\gb=-1$.
If $\charc(\ff)=3,$ then we get $\ga+\gb=1$.

If $\charc(\ff)\ne 3,$ then $(-a-b)(a-b)=-(a+b)(a-b)=-(a-b),$
so $-a-b$ is a $-1$-axis.

Next if $\charc(\ff)=3,$ then
\begin{gather*}
(\ga a+(1-\ga)b)(a-b)=\ga a -\ga ab+(1-\ga)ab-(1-\ga)b\\
=\ga a+\ga a+\ga b-a-b+\ga a+\ga b-b+\ga b=\\ 
-a-2b=-a+b=-(a-b).\qedhere
\end{gather*}
\end{proof}

\begin{lemma}\label{lem jab}
Let $a,b\in\calx$ such that $N:=N_{a,b}=J_{a,b}$ $($so that $\eta=\half)$.
Let $\{x,y\}=\{a,b\},$   then
\begin{enumerate}
\item
$\gs_{a,b}=0,$ $\pi_{a,b}=0$ and $\gvp_{a,b}=1;$
\item
$ab=\half a +\half b;$

\item
$N_{\half}(x)=\ff(x-y);$

\item
$x^{\tau(y)}=2y-x;$

\item
$x^{(\tau(y)\tau(x))^k}=(2k+1)x-2ky,$ for all $k\ge 0.$
\end{enumerate}
\end{lemma}
\begin{proof}
Parts (1) and (2) are by definition.  We have $a(a-b)=a-ab=a-(\half a+\half b)=\half a-\half b$ 
and similarly for $b,$ so (3) holds.
Since $a=b+(a-b),$ then, by definition, $a^{\tau(b)}=b-(a-b)=2b-a,$ and similarly for $b$.
This shows (4).  We leave the calculations of (5) to the reader.
\end{proof}

\begin{lemma}\label{lem etanothalf}
Suppose that $\eta=\half,$ and let $a, b\in\calx$ with $a\ne b$. 
Set $\gvp=\gvp_{a,b}$ and $\gs=\gs_{a,b},$ then
\[
b^{\tau(a)}=-4\gs-(2-4\gvp)a-b.
\]
\end{lemma}
\begin{proof}
Suppose  $\gs\ne 0$. If $N_{a,b}\ne 2B_{a,b},$ this follows from Lemma \ref{lem 2dim} and Theorem \ref{thm betaphi}(6).
If $N_{a,b}= 2B_{a,b},$ then $\gvp=0$ and $\gs=-\half(a+b),$
so $b^{\tau(a)}=2(a+b)-2a-b=b$. Suppose $\gs=0,$  
then this follows from Lemma \ref{lem 2dim}, Lemma \ref{lem jab} and from Lemma \ref{lem 3c(-1)times}
when $\charc(\ff)=3$.
\end{proof}

\subsection{Some further consequences}\label{sub cons}\hfill
\medskip

\noindent
In this subsection we derive further properties of the algebras discussed in subsection
\ref{sub details}.
 
\begin{lemma}\label{lem 2b}
Let $a, b$ be two distinct $\eta$-axes in $A$.
Then $N_{a,b}=2B_{a,b}$ if and only if $a^{\tau(b)}=a$.
In particular, if $\tau(a)=\tau(b),$ then $N_{a,b}=2B_{a,b}$.
\end{lemma}
\begin{proof}
We have $a^{\tau(b)}=a$ if and only if
\[\tag{$*$}
\text{the projection of $a$ into the $\eta$-eigenspace of ${\rm ad}_b$ is $0$.}
\]
Now if $N_{a,b}=2B_{a,b},$ then $ab=0,$ so clearly $(*)$ holds.  
If $(*)$ holds then, by \cite[Proposition 2.8, p.~88]{HRS}, $ab=0,$
so $N_{a,b}=2B_{a,b}$. 

Next, if $\tau(a)=\tau(b),$ then $a^{\tau(b)}=a^{\tau(a)}=a,$ so $N_{a,b}=2B_{a,b}$.
\end{proof}

\begin{lemma}\label{lem vsubb}
Let $a, b$ be two distinct $\eta$-axes in $A$ and
assume that $N:=N_{a,b}=B(\eta,\gvp_{a,b})_{a,b}$ is $3$-dimensional.
Set $\gs=\gs_{a,b},\ \gvp=\gvp_{a,b}$ and $\pi=\pi_{a,b}$.
Set $V:=N_{b, b^{\tau(a)}}$. Then either $V=N$ or $N$ has an identity
$1_{a,b}$ and exactly one
of the following holds:
\begin{itemize}
\item[(i)]
$N=B(\half, 0)_{a,b},\ V=J_{b,b^{\tau(a)}}$ and then ${\rm Span}\{b, b^{\tau(a)}\}\cap {\rm Span}\{1_{a,b}, a\}=\ff(1_{a,b}-a)$.
\textsl{}
\item[(ii)]
$N=B(\half,\half)_{a,b},$\  $V=2B_{b,b^{\tau(a)}}$ and then ${\rm Span}\{b, b^{\tau(a)}\}\cap {\rm Span}\{1_{a,b}, a\}=\ff 1_{a,b}$.
\end{itemize}
\end{lemma}
\begin{proof}
Assume that $V$ is $2$-dimensional.
By Lemma \ref{lem 2dim} we must consider $2$ cases.
\medskip

\noindent
\underline{ Case 1}.    $\eta\in\{-1,\half\}$ and $bb^{\tau(a)}=\eta b^{\tau(a)}+\eta b$.
\smallskip

\noindent
In this case, by Theorem \ref{thm betaphi}(6)
\begin{gather*}\textstyle{
bb^{\tau(a)}=\eta b^{\tau(a)}+\eta b=-\frac{2\eta}{\eta}\gs-\frac{2\eta(\eta-\gvp)}{\eta}a-\eta b+\eta b}\\
=-2\gs-2(\eta-\gvp)a.
\end{gather*}
Hence by Theorem \ref{thm betaphi}(7),
\begin{gather*}
 -2\gs-2(\eta-\gvp)a\\\textstyle{
=-\frac{2(\eta-\gvp)}{\eta}\gs-\left(2(\eta-\gvp)+\frac{2}{\eta}\pi+1\right)b-2(\eta-\gvp)a.}
\end{gather*}
We conclude that $2=\frac{2(\eta-\gvp)}{\eta}$.  This implies
\[
\gvp=0\quad\text{and}\quad\pi=-\eta.
\]
But also $2\eta+\frac{2}{\eta}\pi+1=0,$ thus
\[\textstyle{
\pi=-\frac{2\eta^2+\eta}{2},\quad\text{ so for }\eta\in\{-1,\half\},\ \pi=-\half.}
\]
Note that $-1=\half$ if $\charc(\ff)=3,$ hence the only case that can occur is (i).
Also by Lemma \ref{lem etanothalf} we have
\[\textstyle{
-\half b+\half b^{\tau(a)}=-\half b+\half(-4\gs-2a-b)=-2\gs-a=1_{a,b}-a.}
\]
\medskip

\noindent
\underline{Case 2}.  $bb^{\tau(a)}=0$.
\smallskip

\noindent
In this case, by Theorem \ref{thm betaphi}(7),
\[\textstyle{
-\frac{2(\eta-\gvp)}{\eta}\gs-\left(2(\eta-\gvp)+\frac{2}{\eta}\pi+1\right)b-2(\eta-\gvp)a=0.}
\]
Hence $\eta=\gvp$ and $\pi=-\frac{\eta}{2}$.  Since $\pi=(1-\eta)\gvp-\eta$ we get
$(1-\eta)\eta=\frac{\eta}{2},$ and $\eta=\half$.  This is case (ii).
Since $b^{\tau(a)}=1_{a,b}-b$ the last claim of (ii) holds.
\end{proof}

\begin{lemma}\label{lem identity}
Suppose there exists an element $\e\in A$ such that $\e^2=\e$
and $\e\cdot a=a,$
for all $a\in\cala$.  Then $\e$ is the identity element of $A$.
\end{lemma}
\begin{proof}
Since $\e=a+(\e-a)\in A_1(a)+A_0(a),$ by definition, $\e^{\tau(a)}=\e$
for all $a\in \cala$.
Let $G:=\lan \tau(a)\mid a\in\cala\ran$.
Then $\e^g=\e,$ for all $g\in G$.  Since any $x\in [\cala]$
has the form $a^g,$ for some $g\in G$ and $a\in\cala,$ we see that $\e\cdot x=\e,$ for
all $x\in [\cala]$. 
By \cite[Corollary 1.2]{HRS}, $A$ is spanned by $[\cala],$
so $\e$ is the identity of $A$.
\end{proof}

\begin{lemma}\label{lem ta}
Assume that $\eta=\half$ and that
$a,b,c\in\calx$ are distinct.  Suppose there exists an element $\e\in A$
such that $\e^2=\e$ and $\e x=x,$ for all $x\in\{a,b,c\}$.  Suppose further that  
\[\textstyle{
xy=\ga_{x,y}\e+\half x+\half y,} 
\]
where $\ga_{x,y}\in\ff,$ for all distinct $x,y\in\{a,b,c\}$.
Then
\begin{enumerate}
\item
$\ga_{x,y}=\pi_{x,y},$ for all distinct $x, y\in\{a,b,c\};$

\item
$ab^{\tau(c)}=\big(8\pi_{a,c}\pi_{b,c}+2\pi_{a,c}-\pi_{a,b}+2\pi_{b,c}\big)\e+\half a+\half b^{\tau(c)}.$
\end{enumerate}
\end{lemma}
\begin{proof}
(1):\quad Let $x,y\in\{a,b,c\}$ with $x\ne y$.  Set $N:=N_{x,y}$.  Suppose first that  $\dim(N)=2$.
If $xy=0,$ then  $xy=-\half (x+y)+\half x+\half y,$ and by Lemma \ref{lem 2dim}(1), $\pi_{x,y}=-\half$.
Also $x+y$ is the identity element of $N$.  Hence $\e=x+y$ and $\ga_{x,y}=\pi_{x,y}=-\half$.
Otherwise, since $N$ contains no identity element, $\ga_{x,y}=\pi_{x,y}=0,$ by Lemma \ref{lem 2dim}(2\&3).

Suppose $\dim(N)=3$.  By Theorem \ref{thm betaphi}(4), $xy=\gs_{x,y}+\half x+\half y$
and $\gs_{x,y}\ne 0$.
Hence $\gs_{x,y}=\ga_{x,y}\e\ne 0$.  Thus $\e\in N$ so by Lemma \ref{lem identity}, $N$ contains
an identity element $\e$.  But if $\pi_{x,y}=0,$ then $N$ contains no identity element (Theorem \ref{thm betaphi}(3)).
Hence $\pi_{x,y}\ne 0,$ and then $1_{x,y}=\e$.  But $\gs_{x,y}=\pi_{x,y}1_{x,y}=\pi_{x,y}\e$.  
Hence $\ga_{x,y}=\pi_{x,y}$.
\medskip

\noindent
(2):\quad
Set
\[
\ga:=\pi_{a,b},\quad  \gb:=\pi_{b,c}\quad\gc:=\pi_{a,c}\quad\text{ and }\gvp:=\gvp_{b,c} 
\]
By Lemma \ref{lem etanothalf} (and since here $\gs_{b,c}=\pi_{b,c}\e$),
\[
b^{\tau(c)}=-4\gb\e-(2-4\gvp)c-b.
\]
Hence (recalling that $\gb=\pi_{b,c}=\half\gvp-\half)$ we have
\begin{align*} 
ab^{\tau(c)}&=a(-4\gb\e-(2-4\gvp)c-b)=-4\gb a-(2-4\gvp)ac-ab\\
&=\textstyle{-4\gb a+(4\gvp-2)(\gc\e+\half a+\half c)-(\ga\e+\half a+\half b)}\\
&=\textstyle{-4\gb a+(4\gvp-2)\gc\e+(2\gvp-1) a+(2\gvp-1) c-\ga\e-\half a-\half b}\\
&=\textstyle{((4\gvp-2)\gc-\ga)\e+(2\gvp-\frac{3}{2}-4\gb)a+(2\gvp-1)c-\half b}\\
&=\textstyle{((4\gvp-2)\gc-\ga+2\gb)\e+(2\gvp-\frac{3}{2}-4\gb)a-2\gb\e+(2\gvp-1)c-\half b}\\
&=\textstyle{((4\gvp-2)\gc-\ga+2\gb)\e+\half a+\half b^{\tau(c)}.}
\end{align*}
Since $\gb=\half\gvp-\half$ we get that $\gvp=2\gb+1$.
\end{proof}

\begin{lemma}\label{lem ab=0}
Let $a,c$ be two distince $\eta$-axes in $A$ and assume that
there exists an $\eta$-axis $b\in N_{a,c}$ such that $ab=0$.
Then either $N_{a,c}=2B_{a,c}$ and $b=c,$ or $\eta=\half$ and
$N_{a,c}$ contains an identity $1_{a,c}=a+b$. 
\end{lemma}
\begin{proof}
If $b=c,$ this is clear.  So assume $b\ne c$.
Suppose that $N_{a,c}$ is $2$-dimensional.
Then $N_{a,c}=N_{a,b}=2B_{a,b},$ so $b=c,$ a contradiction.

Hence $N:=N_{a,c}$ is $3$-dimensional. Set $B:=N_{a,b}=2B_{a,b}$. 
Let $\gs=\gs_{a,c}\ne 0$.  
If $\pi_{a,c}=0$, then $\gs\notin B$
because $\gs$ is in the annihilator of $N$ and 
$B$ has $\{0\}$ annihilator.  But this implies that $N=B\oplus\ff\gs$
is associative, contradicting the fact that $\eta$ is
an eigenvalue of ${\rm ad}_a$ (and ${\rm ad}_c$). (The eigenvalues of idempotents
in an associative algebra are $0$ and $1$.)

Thus $\pi_{a,c}\ne 0$ and then $1_{a,c}=\frac{1}{\pi_{a,c}}\gs$ is the identity element of $N$.
Now if $1_{a,c}\notin B,$ then again we get that $N=B\oplus \ff\cdot 1_{a,c}$ is
associative a contradiction.  Hence $1_{a,c}\in B,$ so $1_{a,c}=a+b$.
Also, since $a=1_{a,c}-b$ is an $\eta$-axis in $A$ and it has eigenvalues $0,1,1-\eta,$
we must have $\eta=\half$.
\end{proof}

The following properties of dihedral groups are well-known and easy to check:

\begin{lemma}\label{lem t s}
Let $D:=\lan t, s\ran$ be a dihedral group such
that $t,s$ are involutions and 
such that $|ts|=k \ge 2$.  Then
\begin{enumerate}
\item
if $k$ is odd then $t^{(st)^{\frac{k-1}{2}}}=s$ and $s^{(ts)^{\frac{k-1}{2}}}=t;$
\item
if $k$ is even then
\begin{enumerate}
\item
$t^{(st)^{\frac{k-2}{2}}s}=t;$

\item
\begin{equation*}
(st)^{\frac{k-2}{2}}s=
\begin{cases}
s^{(ts)^{\frac{k-2}{4}}}, &\ {\rm if}\   k\equiv 2\ ({\rm mod}\ 4)\\
t^{(st)^{\frac{k-4}{4}}s}, &\ {\rm if}\ k\equiv 0\ ({\rm mod}\ 4).
\end{cases}
\end{equation*}
\end{enumerate}
\end{enumerate}
\end{lemma}

\begin{lemma}\label{lem tatbk}
Let $a, b$ be two distinct  $\eta$-axes in $A$ and assume that $|\tau(a)\tau(b)|=k<\infty$.  
Set $t=\tau(a)$ and $s=\tau(b)$ and assume that $s\ne t$.  Let $N=N_{a,b}$.
Then,
\begin{enumerate}
\item
if $k$ is odd then either $a=b^{(ts)^{\frac{k-1}{2}}},$
or $\eta=\half$ and $N_{a,b}$ contains an identity $1_{a,b}$.
Further $1_{a,b}-a=b^{(ts)^{\frac{k-1}{2}}},$ $\tau(1_{a,b}-a)=\tau(a)$
and $\tau(1_{a,b}-b)=\tau(b).$

\item
If $k=2$ then either $N=2B_{a,b}$ or $N=B(\half,\half)_{a,b}$ and 
$\tau(1_{a,b}-x)=\tau(x)$ for $x\in\{a,b\}.$

\item 
If $k\ge 4$ is even then $\eta=\half,$ $N$ contains an identity $1_{a,b},$ and $(1_{a,b}-x)\in\calx$ for $x\in\{a,b\}$.
Furthermore
\begin{enumerate}
\item
If $k\equiv 2\ {\rm mod}(4)$ then either $\tau(1_{a,b}-x)\ne \tau(x),$ for $x\in\{a,b\}$ and
\[
a+b^{(ts)^{\frac{k-2}{2}}}=1_{a,b}=b+a^{(st)^{\frac{k-2}{2}}}
\]
or for $x\in\{a,b\}$ there exists $c_x\in\calx\cap N$ such that $N=B(\half,\half)_{x,c_x},$
$\tau(x)=\tau(1_{a,b}-x)$ and 
\[
a+a^{(st)^{\frac{k-2}{2}}s}=1_{a,b}=b+b^{(ts)^{\frac{k-2}{2}}t}.
\]
\item
If $k\equiv 0\ {\rm mod}(4),$ then $\eta=\half$ and for $x\in\{a,b\}$ there exists $c_x\in\calx\cap N$ 
such that $N=B(\half,\half)_{x,c_x}$.  Also 
\[
a+a^{(st)^{\frac{k-2}{2}}s}=1_{a,b}=b+b^{(ts)^{\frac{k-2}{2}}t},
\]
$\tau(1_{a,b}-a)\ne\tau(a)$ and $\tau(1_{a,b}-b)\ne\tau(b)$.
\end{enumerate}
\end{enumerate}
\end{lemma}
\begin{proof}
(1):\quad
Assume that $k$ is odd. Set $g=(st)^{\frac{k-1}{2}}$.  By Lemma \ref{lem t s}
\[
\tau(a^g)=\tau(b).
\]
By Lemma \ref{lem 2b}, either $a^g=b$
or $a^gb=0$.  Suppose $a^gb=0$. Since $a^g\ne a$ (because $ab\ne 0$),
Lemma \ref{lem ab=0} implies that
$\eta=\half$ and $N_{a,b}$ contains an identity $1_{a,b}=a^g+b$.
Hence $1_{a,b}-b=a^g,$ so $\tau(1_{a,b}-b)=\tau(a^g)=\tau(b)$.  
Similarly $\tau(1_{a,b}-a)=\tau(a)$.
\medskip

\noindent
(2):\quad
Assume that $k$ is even.
Let 
\[
g=(st)^{\frac{k-2}{2}}s.
\]
By Lemma \ref{lem t s}, $t^g=t,$ so 
\[
\tau(a^g)=\tau(a),
\]
and
\begin{equation*}
g=
\begin{cases}
\tau(b^h), &\ {\rm if}\   k\equiv 2\ ({\rm mod}\ 4)\\
\tau(a^h), &\ {\rm if}\ k\equiv 0\ ({\rm mod}\ 4),
\end{cases}
\end{equation*}
where
\begin{equation*}
h=
\begin{cases}
(ts)^{\frac{k-2}{4}}, &\ {\rm if}\   k\equiv 2\ ({\rm mod}\ 4)\\
(st)^{\frac{k-4}{4}}s, &\ {\rm if}\ k\equiv 0\ ({\rm mod}\ 4).
\end{cases}
\end{equation*}
By Lemma \ref{lem 2b}, either $a^g=a$ or $aa^g=0$.  
Assume first that $a^g=a$.  If $k=2,$ then $a^{\tau(b)}=a$ so
by Lemma \ref{lem 2b}, $N_{a,b}=2B_{a,b}$. So let $k\ge 4$.
We  have
\begin{equation*}
a=
\begin{cases}
a^{\tau(b^h)}, &\ {\rm if}\   k\equiv 2\ ({\rm mod}\ 4)\\
a^{\tau(a^h)}, &\ {\rm if}\ k\equiv 0\ ({\rm mod}\ 4).
\end{cases}
\end{equation*}
Since $a\notin \{a^h, b^h\}$ in the respective cases  
(because $t\notin \{t^h,s^h\}),$    
Lemma \ref{lem 2b} and Lemma \ref{lem ab=0} imply that
\begin{equation*}
1_{a,b}=
\begin{cases}
a+b^h, &\ {\rm if}\   k\equiv 2\ ({\rm mod}\ 4)\\
a+a^h, &\ {\rm if}\ k\equiv 0\ ({\rm mod}\ 4)
\end{cases}
\end{equation*}
If $k\equiv 2\ {\rm mod}(4),$ then $b^h=1_{a,b}-a$ and $\tau(1_{a,b}-a)\ne\tau(a)$.
Similarly $a^{h^{-1}}=1_{a,b}-b$ so $\tau(1_{a,b}-b)\ne\tau(b).$

If $k\equiv 0\ {\rm mod}(4),$ then $a^h=1_{a,b}-a$.  Since $h=\tau(c)$ for some $c\in\calx\cap N_{a,b},$ 
we see that  $N=B(\half,\half)_{a,c}$.  Further $\tau(a)\ne\tau(1_{a,b}-a)$.
As we will see later (see Theorem \ref{thm ta=tb})
this also yields $\tau(1_{a,b}-b)\ne\tau(b)$

Assume next that
$aa^g=0$.  
Note that $a^g\ne b$ since otherwise $ab=0,$ and then $a^s=a$
and it would follow that $a^g=a,$ a contradiction.  By Lemma \ref{lem ab=0}
we get  that $N_{a,b}$ contains an identity $1_{a,b}=a+a^g$.  Then
$\tau(1_{a,b}-a)=\tau(a^g)=\tau(a)$.  Since $g=\tau(b^h)$ or $g=\tau(a^h),$
and since $a^g=1_{a,b}-a,$  Lemma \ref{lem etanothalf} implies that $N=B(\half,\half)_{a, c_a},$
with $c_a\in\{a^h, b^h\}$.  The argument above (i.e.~the case $a^g=a$) shows
that necessarily the roles of $a$ and $b$ can be interchanged
(since $b^g=b$ implies that $\tau(x)\ne\tau(1_{a,b}-x),$ for $x\in\{a,b\}$).
So Parts (1) and (2) of the Lemma hold in case $aa^g=0,$
and the proof of the lemma is complete.
\end{proof}
 
\begin{lemma}\label{lem closure}
Let $\calb\subseteq\calx$.  Then
\begin{enumerate}
\item
$[\calb]=\bigcup_{g\in G_{\calb}}\calb^g;$

\item
for each $g\in G_{\calb}$ there are $g_1, g_2,\dots g_k\in G_{\calb}$  $(k\ge 2),$
with $g_1={\rm id}$ and $g_k=g,$
such that $\calb^{g_i}\cap\calb^{g_{i+1}}\ne\emptyset,$
for all $i=1,\dots k-1,$ in particular;

\item
for each $x\in[\calb]$ there are $g_1,\cdots,g_k$ as in (2),
such that $x\in\calb^{g_k};$

\item
$G_{\calb}=G_{[\calb]}$ and $G_{\calb}=G_{\calb^{\eta}};$

\item
$[B^{\eta}]=[\calb]^{\eta};$

\item
if $a\in\calx^1,$ then $ax=0$ and $a^{\tau(x)}=a,$ for all $x\in\calx\sminus\{a\};$

\item
$[\calb]=\calb^1\cup[\calb^{\eta}]$ a disjoint union and $[\calb]^1=\calb^1;$

\item
$N_{\calb}=N_{[\calb]}.$
\end{enumerate}
\end{lemma}
\begin{proof}
(1):\quad Set $G:=G_{\calb}$.
Let $\calc:=\bigcup_{g\in G}\calb^g$.  Clearly $\calc\subseteq [\calb]$.
Let now $c\in\calc$.  Then $c\in\calb^g$ for some $g\in G$.
Thus $\tau(c)=\tau(b^g)=\tau(b)^g,$ for some $b\in\calb$.  But then
$\tau(c)\in G,$ so $\calc^{\tau(c)}\subset \calc$.  Thus $\calc$
is closed, so $\calc=[\calb]$.
\medskip

\noindent
(2):\quad
Let $g\in G$ and
write $g=\tau(b_1)\tau(b_2)\cdots\tau(b_m),$ with $b_i\in\calb$ for all $i$.
We prove (2) by induction on $m$.  If $m=1,$ then $b_1\in \calb\cap\calb^{\tau(b_1)}$
so (2) holds.  Next let $h:=\tau(b_1)\cdots\tau(b_{m-1}),$ and let ${\rm id}=h_1,\dots,h_k=h$
with  $\calb^{h_i}\cap\calb^{h_{i+1}}\ne\emptyset,$
for all $i=1,\dots k-1.$
Then $b_m\in\calb\cap\calb^{\tau(b_m)}$ and letting $g_{i+1}=h_i\tau(b_m),$
$i=1,\dots k$ and $g_1={\rm id}$ we have $g_{k+1}=g,$ and clearly
(2) hold for $g$ (and $k+1$).
\medskip

\noindent
(3):\quad  This is immediate from (1) and (2).
\medskip

\noindent
(4):\quad  Let $x\in[\calb]$. By (1),
$x=b^g,$ for some $b\in\calb$ and some $g\in G_{\calb},$
so $\tau(x)=\tau(b)^g\in G_{\calb}$.  Hence $G_{[\calb]}\le G_{\calb}$.
Also, since (by definition) $\tau(a)={\rm id},$ for $a\in\calb^1,$ it is clear that $G_{\calb}=G_{\calb^{\eta}}$. 
\medskip

\noindent
(5):\quad Since $\calb^{\eta}\subseteq\calb,$ we have $[\calb^{\eta}]\subseteq[\calb]$.  Let
$a\in[\calb^{\eta}]$. By (1), there is $b\in\calb^{\eta}$ and $g\in G_{\calb^{\eta}}$ such that 
$a=b^g$.  Since $\tau(b)\ne {\rm id},$ also $\tau(a)=\tau(b)^g\ne{\rm id}$.  Hence $a\in[\calb]^{\eta}$.

Let $a\in [\calb]^{\eta},$ by (1) there exists $b\in\calb$ and $g\in G_{\calb}=G_{\calb^{\eta}}$ (by (4))
such that $a=b^g$.  Since $\tau(a)\ne{\rm id}$ it follows that $\tau(b)\ne{\rm id}$ so $b\in\calb^{\eta}$
and we see that $a\in[\calb^{\eta}].$
\medskip

\noindent
(6):\quad
By definition $\tau(a)={\rm id},$ so $x^{\tau(a)}=x,$ for all $x\in\calx$.
Hence (6) follows from Lemma \ref{lem 2b}.
\medskip

\noindent
(7):\quad
Clearly $\calb^1\cup[\calb^{\eta}]\subseteq [\calb]$.
Let $a\in[\calb]$.  Using (1) and (4) write $a=b^g,$ with $b\in\calb$
and $g\in G_{\calb^{\eta}}$.  If $b\in\calb^1,$ then by (6), $b^g=b$.
Otherwise $b\in B^{\eta}$ and then $b^g\in [\calb^{\eta}]$.  By
(5) the union is disjoint and $[\calb]^1=\calb^1$.
\medskip

\noindent
(8):\quad
Clearly $N_{\calb}\subseteq N_{[\calb]}$.  Let $b\in\calb$.  Since $N_{\calb}$ is a subalgebra
of $A$ it is invariant under the adjoint action ${\rm ad}_b;$ that is ${\rm ad}_b$ is a linear
transformation of $N_{\calb}$.  Since ${\rm ad}_b$ is semi-simple on $A,$ it is semi-simple
on $N_{\calb}$.  By the definition of $\tau(b)$ it follows that $N_{\calb}$ is $\tau(b)$-invariant.
As this holds for all $b\in\calb$ we see that $N_{\calb}$ is $G_{\calb}$-invariant.
By (1), since $\calb\subseteq N_{\calb}$ also $[\calb]\subseteq N_{\calb}$ so $N_{[\calb]}\subseteq N_{\calb}$. 
\end{proof}

\subsection{Properties related to $3$-transpositions}\hfill
\medskip

\noindent
This subsection is devoted to results related to the question of when an axial algebra 
is a $3$-transposition algebra with respect to a generating set of $\eta$-axes
(see subsection \ref{sub intro 3-trans} of the introduction for a definition).
These results will be applied in \S\ref{sect 3trans}.

\begin{lemma}\label{lem 2dim3}
Let $a,b$ be two distinct  $\eta$-axes in $A$ and suppose that
$N_{a,b}$  is $2$-dimensional.  Assume further that $|\tau(a)\tau(b)|\in \{2, 3\}$.
Then either   $N_{a,b}=3C(-1)^{\times}_{a,b},\ |\tau(a)\tau(b)|=3,$
and $a^{\tau(b)}=b^{\tau(a)}$. Or $|\tau(a)\tau(b)|=2$ and $N_{a,b}=2B_{a,b}.$
\end{lemma}
\begin{proof}
We use Lemma \ref{lem 2dim}. Set $N=N_{a,b}$.   
If $N$ is as in Lemma \ref{lem 2dim}(1c), then by Lemma \ref{lem jab},
$a^{(\tau(a)\tau(b))^2}=5a-4b$ and $a^{(\tau(a)\tau(b))^3}=7a-6b$.   
Hence $|\tau(a)\tau(b)|\ne 2$.  If $\charc(\ff)\ne 3,$ then it follows that $|\tau(a)\tau(b)|\ne 3,$ while
if $\charc(\ff)=3,$ then $N=3C(-1)^{\times}_{a,b}$ and then 
by Lemma \ref{lem 2dim}, $a^{\tau(b)}=b^{\tau(a)}.$

If $N_{a,b}=3C(-1)^{\times}_{a,b},$ then again $a^{\tau(b)}=b^{\tau(a)}$.
Finally if $N_{a,b}=2B_{a,b},$ then, by Lemma \ref{lem 2b}, $|\tau(a)\tau(b)|=2$. 
\end{proof}
 
\begin{cor}\label{cor halfeta}
Let $a, b$ be two distinct $\eta$-axes in $A,$ and assume that $a^{\tau(b)}=b^{\tau(a)}$.
Then $\gvp_{a,b}=\half\eta$ and one of the following holds:
\begin{itemize}
\item[(i)]
$N_{a,b}$ is $2$-dimensional and $N_{a,b}=3C(-1)^{\times}_{a,b}$.

\item[(ii)]
$N_{a,b}$ is $3$-dimensional and $N_{a,b}=3C(\eta)_{a,b}$.
\end{itemize}  
\end{cor}
\begin{proof}
Set $\gvp:=\gvp_{a,b}$.
By Lemma 4.1 and Lemma 4.4 in \cite{HRS},
\[\textstyle{
\gvp b+\frac{\eta}{2}(a-a^{\tau(b)})=\gvp a+\frac{\eta}{2}(b-b^{\tau(a)}),}
\]
Hence $0=\gvp(a-b)+\frac{\eta}{2}(b-a),$ so since $a\ne b,$ we have $\gvp=\frac{\eta}{2}$.
Notice that $|\tau(a)\tau(b)|=3$.
If $N_{a,b}$ is $2$-dimensional, then (i) follows from Lemma \ref{lem 2dim3}.
If $N_{a,b}$ is $3$-dimensional,
then $N_{a,b}=B(\eta,\, \half\eta)_{a,b},$ so, by definition, $N_{a,b}=3C(\eta)_{a,b}$.
\end{proof}

\begin{lemma}\label{lem 3chalfab}
Let $\eta=\half$ and let $a, b\in A$ be two distinct $\half$-axes.
Set $N:=N_{a,b}$ and assume that $N$ is contained in a $3$-dimensional
subalgebra $M$ of $A$ such that $M$ contains an identity element $\e$.
Then the following are equivalent:
\begin{itemize}
\item[(i)]
$\charc(\ff)\ne 3$  
 and $N=3C(\half)_{a,b},$ or $\charc(\ff)=3$ and $N=3C(-1)^{\times}_{a,b}.$

\item[(ii)]
$ab=-\frac{3}{8}\cdot \e+\half a+\half b$.

\item[(iii)]
$a^{\tau(b)}=b^{\tau(a)}$.
\end{itemize}
If these conditions hold then
\begin{enumerate}
\item
$\gvp_{a,b}=\frac{1}{4},\ \pi_{a,b}=-\frac{3}{8}$ and  $\gs_{a,b}=-\frac{3}{8}\e;$

\item
$|\tau(a)\tau(b)|=3;$

\item
if $\charc(\ff)\ne 3,$ then $\e=1_{a,b};$

\item
if $c:=\e-b$ is a $\half$-axis in $A$ then $N_{a,c}$ is $3$-dimensional and
\[\textstyle{
ac=-\frac{1}{8}\e+\half a+\half c.}
\]
Also if $\charc(\ff)\ne 3,$ then $N_{a,b}=N_{a,c},$ 
while if $\charc(\ff)=3$ then $N_{a,b}\subsetneqq N_{a,c}$
and $N_{a,c}=B(-1,0)=B(\half,0)$.
\end{enumerate}
\end{lemma}
\begin{proof}
(i)$\iff$(ii):\quad   
If (i) holds and $\charc(\ff)\ne 3,$ then $N=M$ is\linebreak
 $3$-dimensional, $1_{a,b}\ne 0,$ and    
by Lemma \ref{lem 3ceta}, $\pi_{a,b}=-\frac{3}{8}$ and of course $1_{a,b}=\e,$ so
(ii) holds.  If $\charc(\ff)=3,$ then (ii) holds by the definition
of $3C(-1)^{\times}_{a,b}$.

Assume that (ii) holds.  If $\charc(\ff)=3$ then clearly (i) holds.
Suppose $\charc(\ff)\ne 3$. Then $\e\in N_{a,b}$ so   $\e=1_{a,b}$. 
Now Theorem \ref{thm betaphi}(4) shows that $\pi_{a,b}=-\frac{3}{8}$
and that $\pi_{a,b}=\half\gvp_{a,b}-\half$.  Hence $\gvp_{a,b}=\frac{1}{4}$
and $N=B(\half,\frac{1}{4})_{a,b}=3C(\half)_{a,b},$ by Remark
\ref{rem 3ceta}.  Hence (i) holds.
\medskip

\noindent
(i)$\iff$(iii):\quad Suppose (i) holds.  If $\charc(\ff)\ne 3,$  then by Lemma \ref{lem 3ceta}(1d)
(iii) holds.  If $\charc(\ff)=3$ then (iii) holds by Lemma \ref{lem 3c(-1)times}.
If (iii) holds then by Corollary \ref{cor halfeta}, (i) holds. Note that 
when $\charc(\ff)=3,$ $N$ cannot be $3$-dimensional since $3C(-1)_{a,b}$
does not contain an identity element (Lemma \ref{lem 3ceta}(2)).

Part (1) follows from Lemma \ref{lem 3ceta} and part (2) is
an immediate consequence of (iii).  We already saw that (3) holds. To see (4) we have
\begin{gather*}\textstyle{ 
ac=a(\e-b)=a-ab=a-(-\frac{3}{8}\cdot \e+\half a+\half b)=\frac{3}{8}\cdot \e+\half a-\half b}\\
\textstyle{=\frac{3}{8}\cdot \e+\half a-\half \e+\half \e-\half b=-\frac{1}{8}\cdot \e+\half a+\half c.}
\end{gather*}
This shows that $\e\in N_{a,c}$. 
By Lemma \ref{lem 2dim} and (iii),
$N_{a,c}$ is $3$-dimensional and $\e=1_{a,c}$.  Hence if $\charc(\ff)\ne 3$ then $N_{a,c}=N_{a,b}$.
If $\charc(\ff)=3$ then $\pi_{a,c}=-\frac{1}{8}=1=-\half$ and then, by the definition of $\pi_{a,c},$ we have
$\gvp_{a,c}=0$.  This shows  (4).
\end{proof}

\begin{lemma}\label{lem 3chalfa1-b}
Let $\eta=\half$ and let $a, b\in A$ be two distinct $\half$-axes. 
Set $N:=N_{a,b}$. Assume that $\dim(N)=3$ and that $N$ contains
an identity element $1_{a,b}$.  Then the following are equivalent
\begin{itemize}
\item[(i)]
$1_{a,b}-x$ is a $\half$-axis in $A$ and either $\charc(\ff)\ne 3$ and
$N=3C(\half)_{x,\ (1_{a,b}-y)},$ or $\charc(\ff)=3$ and $N_{x,\, (1_{a,b}-y)}=3C(-1)^{\times}_{x,\, (1_{a,b}-y)},$
for $\{x,y\}=\{a,b\}.$

\item[(ii)]
$ab=-\frac{1}{8}\cdot 1_{a,b}+\half a+\half b;$

\item[(iii)]
$a^{\tau(b)}+b^{\tau(a)}=1_{a,b}.$
\end{itemize}
If these conditions hold then
\begin{enumerate}
\item
$\gvp_{a,b}=\frac{3}{4}$ (so $\gvp_{a,b}=0$ if $\charc(\ff)=3$), $\pi_{a,b}=-\frac{1}{8}$ (so
$\pi_{a,b}=1$ if $\charc(\ff)=3$) and  $\gs_{a,b}=-\frac{1}{8}\cdot 1_{a,b}$ (so $\gs_{a,b}=1_{a,b}$ if $\charc(\ff)=3$);
\item
$|\tau(1_{a,b}-x)\tau(y)|=3,$ for $\{x,y\}=\{a,b\};$

\item
$|\tau(a)\tau(b)|\in\{3,6\};$ 

\item
$|\tau(a)\tau(b)|=3$ if and only if $\tau(x)=\tau(1_{a,b}-x)$
for $x\in\{a,b\};$

\item
if $|\tau(a)\tau(b)|=6$ then $\tau(1_{a,b}-x)\tau(x)$
is an involution in the center of $\lan \tau(a),\tau(b)\ran,$ for $x\in\{a,b\}$.
\end{enumerate}
\end{lemma}
\begin{proof}
Set $\pi=\pi_{a,b},\ \gvp=\gvp_{a,b}$ and $\gs=\gs_{a,b}$.
\medskip

\noindent
(i)$\implies$ (ii):\quad
Assume that (i) holds.  By Lemma \ref{lem 3chalfab}, and by the definition 
of $3C(-1)^{\times}_{a,b}$ (when $\charc(\ff)=3$), 
\begin{gather*}\textstyle{
a-ab=a(1_{a,b}-b)=-\frac{3}{8}\cdot 1_{a,b}+\half a+\half(1_{a,b}-b)}\\
\textstyle{=\frac{1}{8}\cdot 1_{a,b}+\half a-\half b},
\end{gather*}
so (ii) holds.  Also, by (i) and Lemma \ref{lem 3chalfab}(iii) (respectively Lemma \ref{lem 3c(-1)times}),
part (2) holds.
\medskip

\noindent
(ii)$\implies$(iii):\quad
By Lemma \ref{lem etanothalf},
\[
x^{\tau(y)}=-4\gs+y-x\quad\text{for }\{x,y\}=\{a,b\}.
\]
Adding we see that $a^{\tau(b)}+b^{\tau(a)}=-8\gs=1_{a,b}$. 
\medskip

\noindent 
(iii)$\implies$(ii):\quad
We know that $N=B(\half,\gvp),$ so by Lemma \ref{lem etanothalf},
\[
x^{\tau(y)}=-4\gs-(2-4\gvp)y-x\quad\text{for }\{x,y\}=\{a,b\}.
\]
Adding we get $a^{\tau(b)}+b^{\tau(a)}=-8\gs-(3-4\gvp)a-(3-4\gvp)b$.
But this expression equals $1_{a,b}$.  Hence $\gvp=\frac{3}{4}$ and $\pi=-\frac{1}{8}$.
Now Theorem \ref{thm betaphi}(4) yields (ii).
\medskip

\noindent
(ii)$\implies$ (i):\quad
Assume that (ii) holds.  By (ii)$\implies$ (iii) we already know that $1_{a,b}-x$
is a $\half$-axis, for $x\in\{a,b\}$.  By Theorem \ref{thm betaphi},  
$\pi=-\frac{1}{8}$ and $\pi=\half\gvp-\half,$ so $\gvp=\frac{3}{4}$.
Now
\begin{gather*}
\textstyle{a(1_{a,b}-b)=a-ab=a-(-\frac{1}{8}\cdot 1_{a,b}+\half a+\half b)=}\\
\textstyle{\frac{1}{8}\cdot 1_{a,b}+\half a-\half b=
\frac{1}{8}\cdot 1_{a,b}+\half a-\half 1_{a,b}+\half 1_{a,b}-\half b}\\
\textstyle{=-\frac{3}{8}\cdot 1_{a,b}+\half a+\half (1_{a,b}- b).}
\end{gather*}
Hence if $\charc(\ff)\ne 3,$ this show that $1_{a,b}$ is in the subalgebra of $N$ generated by
$a$ and $1_{a,b}-b$ and hence $a$ and $1_{a,b}-b$ generate $N$.
Now Lemma \ref{lem 3chalfab} shows that  
that $N=3C(\half)_{a,(1_{a,b}-b)}$.
If $\charc(\ff)=3$ then $N_{a,(1_{a,b}-a)}=3C(-1)^{\times}_{a,(1_{a,b}-b)}$.
By symmetry the same holds for $b$.

We already saw that (1) and (2) hold.  Since $1_{a,b}-b=a^{\tau(b)\tau(a)}$
we have $\tau(a)\tau(1_{a,b}-b)=(\tau(b)\tau(a))^2$.  Hence, by (2), $|\tau(a)\tau(b)|\in\{3,6\}$.
Also $\tau(1_{a,b}-b)\tau(b)=(\tau(a)\tau(b))^3$ so (3), (4) and (5) hold for $b,$ and by symmetry
they also holds for $a$.
\end{proof}

\section{$3$-trans\-pos\-i\-tion groups of $ADE$-type}\label{sect 3transpositions}
\numberwithin{prop}{section}
The purpose of this section is to characterize central quotients
of finite simply-laced Weyl/Coxeter groups of type
$A$, $D$, and $E$ (see Proposition \ref{prop-weyl} for a precise
description of these groups).  Thus we define $3$-transpositions
groups of $ADE$-type (see Definition \ref{def-ade}), and Theorem \ref{thm-ade}
is the main theorem of this section.
In \S 7 we will see how these groups are related to primitive axial algebras of Jordan type half.

We start with a short discussion.
In the $3$-trans\-pos\-i\-tion group $G$, the normal
set of generating $3$-trans\-posi\-tions $D$ is said
to be of
\emph{symplectic type} if for every
$d,e,f \in D$ with $\langle e,f \rangle$ isomorphic
to $\Symm{3}$, the transposition $d$ commutes with
at least one of $\{e,f, efe=fef\}=D \cap \langle e,f \rangle$. 
Equivalently (see \cite{CH,JH1,HSo}) $G$ has no subgroup
$H=\langle D\cap H \rangle$ 
with $|D \cap H|=9$; that is, $|H|
\neq 18,54$.

The name comes from the fact that (see Theorem
\ref{thm-symp-type} below)
every group of symplectic type arises from
a subgroup of a symplectic group over $\ff_2$
that is generated by transvections 
(a generating $3$-trans\-pos\-i\-tion class in the
full symplectic group).

Let us recall the notion of the {\it diagram}:
Given a subset $Y\subseteq D,$
the {\it diagram} of $Y$ is the graph whose vertex set is $Y$
and $a,b\in Y$ form an edge if and only if $|ab|=3$.

It is well-known and easy to see \cite{CH,JH2} that
the finite simply-laced Weyl/Coxeter groups of type
$A$, $D$, and $E$ are $3$-trans\-pos\-i\-tion groups
with the Weyl generators contained in a 
$3$-trans\-pos\-i\-tion class
of symplectic type. 
These facts were of great help in the classification
\cite{CH}
of $3$-trans\-pos\-i\-tion groups with trivial center.
For instance, the diagram $A_3\,(=D_3)$ is complete
bipartite $K_{1,2}$, and
the isomorphism $\Weyl(A_3) \cong \Symm{4}$
leads directly to a result that is  often used
without reference:

\begin{lemma}
\label{lem-center}
Let $G$ be a group generated by the conjugacy 
class $D$ of $3$-trans\-posi\-tions. Then
$D \cap d\oZ(G) = \{d\}$
for each $d\in D$
and
$\oZ(G)=\oZ_2(G)$.
\end{lemma}

\begin{proof}
This is due to Fischer and can be found
as \cite[Lemma~(3.16)]{CH} and
\cite[(4.3)]{JH1} (where the assumption
of symplectic type is not used).

If $G=\langle D \rangle \cong \Symm{2}$,
then this is certainly true. Otherwise there
is an $e \in D$ with $\langle d,e \rangle
\cong \Symm{3}$. Were there to be an $
f \in d\oZ(G)$ with $d\neq f$, then $\{d,e,f\}$ would have
diagram $A_3$ and so generate a subgroup
$H \cong \Symm{4}$. But then $1 \neq df \in \oZ(G)$
while $\oZ(H)=1$. The contradiction shows
that no such $f$ exists.

The subgroup $\oZ(G)$ is clearly the kernel
of the action of $G=\langle D \rangle$ 
on $D$ by conjugation. But the previous
paragraph implies that $\oZ_2(G)$
is also in this kernel. Thus $\oZ(G)=\oZ_2(G)$.
\end{proof}

Let us now define groups of $ADE$-type.

\begin{Def}
\label{def-ade}
In the $3$-trans\-pos\-i\-tion group $G$, the normal
set of generating $3$-trans\-posi\-tions $D$ is said
to be of \emph{$ADE$-type} provided 
it is of symplectic type and there is no subgroup
$H=\langle D \cap H\rangle$ 
isomorphic to a central quotient of $\Weyl_2(\tilde{D}_4)$.
The group $G$ is then called a group of $ADE$-type.

Recall that $\tilde{D}_4$ is the complete bipartite graph $K_{1,4},$
and see Proposition \ref{prop-klein}(4) for $\Weyl_2(\tilde{D}_4)$.
\end{Def}

In this section we will prove:

\begin{thm}
\label{thm-ade}
Let $G$ be a finite $3$-trans\-pos\-i\-tion group
generated by the conjugacy class $D$ of $3$-trans\-pos\-i\-tions
having $ADE$-type. Then there is an $n \in \zz^+$
with $G$ a central quotient of 
$\Weyl(A_n)$ for $n \ge 2$, $\Weyl(D_n)$ for $n\ge 4$,
or $\Weyl(E_n)$ for $n \in \{6,7,8\}$.
All of these groups are of $ADE$-type.
\end{thm}

Given the appropriate definitions,
Theorem \ref{thm-ade} remains true for infinite $3$-trans\-pos\-i\-tion groups
of $ADE$-type. In this paper we are only concerned with the finite
case.

\begin{prop}
\label{prop-weyl}
Let $X$ be a subset of $D$, a normal set
of $3$-trans\-posi\-tions in the group $G$.
Set $H=\langle X \rangle$.
\begin{enumerate}
\item
If $X$ has diagram (isomorphic to) $A_n$ then
$H$ is isomorphic to the Weyl/Coxeter group
$\Weyl(A_n)\cong \Symm{n+1}$.
\item
If $X$ has diagram $D_n$ then
$H$ is isomorphic to a central quotient
of $\Weyl(D_n)$. That is, either
$H \cong\Weyl(D_n)\cong 2^{n-1}\colon\Symm{n}$
or $H \cong \Weyl(D_{2k})/\oZ(\Weyl(D_{2k}))\cong 2^{2k-2}\colon\Symm{2k}$.
\item
If $X$ has diagram $E_6$ then
$H$ is isomorphic to
$\Weyl(E_6)\cong O^-_{6}(2)$.
\item
If $X$ has diagram $E_7$ then
$H$ is isomorphic to a central quotient
of $\Weyl(E_7)$. That is, either
$H \cong\Weyl(E_7)\cong 2 \times Sp_{6}(2)$
or $H \cong \Weyl(E_7)/\oZ(\Weyl(E_7))\cong Sp_{6}(2)$.
\item
If $X$ has diagram $E_8$ then
$H$ is isomorphic to a central quotient
of $\Weyl(E_8)$. That is, either
$H \cong\Weyl(E_8)\cong  2 \cdot O^+_{8}(2)$
or the group $H \cong \Weyl(E_8)/\oZ(\Weyl(E_8))\cong O^+_{8}(2)$.
\end{enumerate}
\end{prop}

\begin{proof}
In each case, $H$ must be a quotient of the related
Weyl/Coxeter group. As the elements of $X$ are distinct,
the only possible kernels for this quotient are central.
\end{proof}

As is noted in \cite{CH},
in each of these $3$-trans\-pos\-i\-tion groups
the $3$-trans\-pos\-i\-tion class is uniquely
determined
except for  $\Weyl(A_5) \cong \Symm{6}$,
$\Weyl(D_{2k})$, and $\Weyl(E_8)$ 
where there are two classes of $3$-trans\-pos\-i\-tions,
exchanged by an outer automorphism (a central automorphism
except in the case of $\Symm{6}$).

\medskip

The simply-laced affine Weyl group
$\Weyl(\tilde{X})$ 
for $X \in \{A_n,D_n,E_n\}$
is the split extension of the corresponding rank $n$
root lattice $\Lambda_X$ by the finite Weyl group
$\Weyl(X)$ \cite[p.~173]{bourbaki}.
These
are not $3$-trans\-pos\-i\-tion groups but become
such if we factor by $2\Lambda_X$ or $3\Lambda_X$;
see again \cite{CH}. 
Indeed the factor
group $\Weyl_2(\tilde{X})=\Weyl(\tilde{X})/2\Lambda_X$
is a finite $3$-trans\-pos\-i\-tion group of
symplectic type. For instance,
$\Symm{4}$ is $\Weyl(A_3)=\Weyl(D_3)$ but it is also
$\Weyl_2(\tilde{A}_2)=2^2\colon \Symm{3}$.
(The diagram $\tilde{A}_2$ is a triangle.)
Here the normal elementary abelian $2^2$
is the mod $2$ root lattice $V_4=\Lambda_{A_2}/2\Lambda_{A_2}$
of type $A_2$, naturally admitting $\Weyl(A_2) \cong \Symm{3}$.

As already mentioned, the diagram $A_3=D_3$ is complete
bipartite $K_{1,2}$. Additionally $D_4$ is $K_{1,3}$ and 
$\tilde{D}_4$ is $K_{1,4}$.

\begin{prop}
\label{prop-klein}
Let $H = (\bigoplus_{i=1}^m V(i))\colon \Symm{3}$ be the
split extension by $\Symm{3}$ of $V=\bigoplus_{i=1}^m V(i)$, a
direct sum of copies $V(i)$ of the $\Symm{3}$-module $V_4$.
\begin{enumerate}
\item $H$ is a $3$-trans\-pos\-i\-tion group, generated
by the class $E=d^H=e^H$ for $\langle d,e \rangle\cong \Symm{3}$,
a complement to $V$. The diagram of $E$ is a complete tripartite
graph $K_{2^m,2^m,2^m}$ with parts $d^V$, $e^V$, and $(ede)^V$.
The group $H$ is generated by $d$ together with a basis of 
the elementary abelian subgroup $\langle e^V \rangle$, this
generating set having diagram $K_{1,m}$.
\item
For $m=1$, the group $H$ is isomorphic to
$\Weyl(A_3)=\Weyl(D_3)\cong\Weyl_2(\tilde{A}_2)$
and is isomorphic to $\Symm{4}$.
\item
For $m=2$, the group $H$ is isomorphic to
the quotient of
$\Weyl(D_4)\cong 2^{1+(2\oplus 2)}\colon \Symm{3}$
by its center of order $2$.
\item
For $m=3$, the group $H$ is isomorphic to
the quotient of
\begin{align*}
\Weyl_2(\tilde{D}_4) & \cong 2^4\colon\Weyl(D_4)\\
&\cong 2^4\colon(2^{1+(2\oplus 2)}\colon \Symm{3})
=(2^3\cdot(2^2\oplus2^2\oplus2^2))\colon\Symm{3}
\end{align*}
by its elementary center of order $2^3$.
\end{enumerate}
\end{prop}

\begin{proof}
The first part is a direct computation. The rest then come
from expanding $\{d,e\}$ to a generating set $X$ from $E$ of size $2+m$
and having the appropriate diagram.
The group
$\Weyl_2(\tilde{D}_4)$ is the group $F(5,24)$ of \cite{HSo}. 
\end{proof}

\begin{prop}
\label{prop-aweyl}
Let $X$ be a subset of $D$, a normal set
of $3$-trans\-posi\-tions of symplectic type in the group $G$.
If $X$ has diagram $\tilde{D}_4$ and $H=\langle X \rangle$
is not a central quotient of $\Weyl(D_4)$, then $H$ is
a central quotient of $\Weyl_2(\tilde{D}_4)$.
\end{prop}

\begin{proof}
The group $H$ must be a quotient of the affine Weyl/Coxeter
group $\Weyl(\tilde{D}_4) \cong \zz^4 \colon \Weyl(D_4)$. 
As it is a $3$-trans\-posi\-tion
group of symplectic type, it is in fact a quotient
of $\Weyl_2(\tilde{D}_4)$. Since it is not a central quotient 
of $\Weyl(D_4)$, the only
possible kernels are central.
\end{proof}

\begin{thm}
\label{thm-symp-type}
Let $G$ be a finite group generated by a conjugacy class 
$D$ of $3$-trans\-pos\-i\-tions of symplectic type.

\begin{enumerate}
\item
There is a normal subgroup $N$ of $G$ such 
that $\bar{G} = G/N$ is
isomorphic to one of the groups 
$\Symm{n}$, $O^\gre_{2m}(2)$, or $Sp_{2m}(2)$
for $4\neq n \ge 2$ and $m \ge 3$ with
$(m,\gre) \neq (3,+)$.
This isomorphism can be chosen to map $D$ to 
the collection of symplectic transvections in $\bar{G}$. 
For $x, y \in D$, $\bar{x} = \bar{y}$ 
if and only if $\oC_D(x) = \oC_D(y)$.
\item
The normal subgroup $[G,N]$ is a $2$-group,
generated by its normal elementary abelian
$2$-subgroups $[x,N]=\langle\, xy\, |\, y \in D\,,\, 
\bar{x} = \bar{y}\,\rangle$ for $x \in D$.
\end{enumerate}
\end{thm}
\begin{proof}
The first part of this theorem is the finite part of
\cite[Theorem~5]{JH1}.
The second part of the theorem then follows
directly from the last sentence of the first
part.
\end{proof}

The restrictions on $m$ in the theorem arise
from isomorphisms of the smaller groups with
certain symmetric groups.

The papers \cite{JH1,JH2} provide a full classification
(up to a central quotient) of all $3$-trans\-pos\-i\-tion groups
of symplectic type, and the paper \cite{CH} describes the
near-complete classification of all $3$-trans\-pos\-i\-tion groups
with trivial center.
In our proof of Theorem 
\ref{thm-ade} we only need the elementary \cite{JH1}, as
detailed in Theorem \ref{thm-symp-type};
in particular
the cohomological arguments of \cite{JH2} are not necessary.

\begin{prop}
\label{prop-d4-free}
Let $G$ be a finite group generated by a conjugacy class 
$D$ of $3$-trans\-pos\-i\-tions of symplectic type.
Assume additionally there is no subgroup
$H=\langle D \cap H\rangle$ 
isomorphic to a central quotient of $\Weyl(D_4)$.
Then there is an $n \in \zz^+$
with $G$ isomorphic to 
$\Weyl(A_n) \cong \Symm{n+1}$ for $n \ge 1$.
\end{prop}

\begin{proof}
This is nearly equivalent to the finite version
of \cite[(2.17)]{JH1}, which
is a step in the proof of \cite[Theorem~5]{JH1}
(the finite version of which is the first part
of Theorem \ref{thm-symp-type}). Here we prove
it as a consequence of Theorem \ref{thm-symp-type}. 

Let $G$, $D$, $N$, and $\bar{G}=G/N$ be as in the previous theorem.

As $D$ is a conjugacy class, if  
$\bar{G} \cong \Symm{2}$ then $G=\bar{G} \cong \Symm{2} \cong \Weyl(A_1)$,
and we are done. So we may assume that there are $a,b$ in 
$D$
with $\langle a,b \rangle \cong \langle \bar{a},\bar{b} \rangle \cong \Symm{3}$.

First suppose that there is a $c \in D$ for which $\{a,b,c\}$
has diagram $A_3$ (with $|ac|=2$ and $|bc|=3$) and additionally that
$\langle \bar{a},\bar{b},\bar{c}\rangle$ is
isomorphic to $\langle a,b,c\rangle$ and hence to $\Weyl(A_3)\cong\Symm{4}$.

As $G$ is generated by $D=a^G$, 
if $[G,N] \neq 1$ then there is an $x \in N$ with
$[a,x]\neq 1$. In that case $a\neq d = x^{-1}ax \in D$
with $\bar{a}=\bar{d}$, so that $\{ a,b,c,d\}$
has diagram $D_4$ and generates a central quotient
of $\Weyl(D_4)$. This contradicts the hypothesis, so
$[G,N]=1$ and $G$ is a central extension of one of the groups of
Theorem \ref{thm-symp-type}.

The groups $O^\gre_{2m}(2)$ and $Sp_{2m}(2)$,
for $m \ge 3$ with
$(m,\gre) \neq (3,+)$ all contain $\Weyl(E_6)
\cong O_6^-(2)$ as a transvection generated subgroup.  
As $E_6$ has $D_4$ as
a subdiagram, these have subgroups
$H=\langle D \cap H\rangle$ that are
isomorphic to a central quotient of $\Weyl(D_4)$
by Proposition \ref{prop-weyl}. 
(Indeed this subgroup is actually $\Weyl(D_4)$.)
Thus the only possibilities
for the quotient $\bar{G}=G/\oZ(G)$ are $\Symm{n}$ for $n \ge 3$.
Such a group $G$ will be generated by a subset of $D$ with diagram
$A_{n-1}$ (by Lemma \ref{lem-center}), and so by Proposition 
\ref{prop-weyl} we have $G \cong \Symm{n}$.

The only groups $\bar{G}$ of the theorem that contain no $\Symm{4}$
are $\Symm{2}$ and $\Symm{3}$. We have already dealt with the
first case, so we may assume
now that $\bar{G}=G/N$ is $\Symm{3}=\langle \bar{a},\bar{b}
\rangle$ for $a,b \in D$. 

If $[G,N]=1$,
then $G = \bar{G} \cong \Weyl(A_2)\cong \Symm{3}$
by Proposition \ref{prop-weyl}. If
$|[G,N]|=2$, then as above there is a 
$d (\neq a)$ with $\bar{a}=\bar{d}$
and $G=\langle a,b,d \rangle \cong \Weyl(A_3) \cong \Symm{4}$.
Finally, if $|[G,N]|>2$ then there are distinct $d,e
(\neq a)$ with $\bar{a}=\bar{d}=\bar{e}$ and 
$\langle a,b,d,e \rangle$ a central quotient
of $\Weyl(D_4)$, against hypothesis.
\end{proof}

\begin{lemma}
\label{lem-o8}
\begin{enumerate}
\item  The $3$-transpositions of $O^-_{8}(2)$ are not of $ADE$-type.
\item The $3$-transpositions of
$O^+_{8}(2)$ are of $ADE$-type.
\end{enumerate}
\end{lemma}

\begin{proof}
(1) $O^-_{8}(2)$  contains a parabolic subgroup
$2^6\colon O^-_6(2)$.
As $ O^-_6(2)\cong\Weyl(E_6)$ contains a
(central quotient of) $\Weyl({D}_4)$
(as mentioned before),
the noncentral extension $2^6\colon O^-_6(2)$
contains a
central quotient of $\Weyl_2(\tilde{D}_4)$
by Proposition \ref{prop-aweyl}.

\medskip
(2) $O^+_{8}(2)$ is of $ADE$-type
if and only if $\Weyl(E_8)$ is (by Lemma
\ref{lem-center}). Suppose
$\Weyl(E_8)$ is not. 
Then it has a subset $S$
of five reflections with diagram $\tilde{D}_4$
that generate $H$, a central quotient of 
$\Weyl_2(\tilde{D}_4)$. In the action of $\Weyl(E_8)$
on $V=\qq^8$ we have $W=[V,H]=[V,S]$ of dimension
at most $5$ and positive definite, as $V$ is.
But by Proposition \ref{prop-aweyl} the reflection
group $H$ contains eight pairwise commuting reflections.
These cannot act on the positive definite space $W$ of
dimension less than $8$, a contradiction.

We conclude that  $\Weyl(E_8)$ and
$O^+_{8}(2)$ are both of $ADE$-type.
\end{proof}

We are now in a position to prove Theorem \ref{thm-ade}.

\begin{proof}[Proof of Theorem \ref{thm-ade}]

Let $G$, $N$, and $\bar{G}=G/N$ be as in Theorem \ref{thm-symp-type},
and assume that the conjugacy class $D$ of $3$-trans\-pos\-i\-tions 
is of $ADE$-type.

As $D$ is a conjugacy class, if 
$\bar{G} \cong \Symm{2}$, then $G=\bar{G} \cong \Symm{2} \cong \Weyl(A_1)$;
and we are done. So we may assume that in $G$ there are $a,
b \in \bar{D}$ with $\langle \bar{a},\bar{b} \rangle\cong 
\langle a,b\rangle \cong \Symm{3}$.

First suppose the normal $2$-group
$[G,N]$ is nontrivial. 
If $\bar{G}$ had a subgroup $\bar{H}=\langle \bar{H} \cap \bar{D}\rangle$
that was a central quotient $\Weyl({D}_4)$, then as $[N,H]
\neq 1$ there would be a new, fifth generator that together with four
lifted from $\bar{H}$ would provide a $\tilde{D}_4$ diagram and 
a central quotient of $\Weyl_2(\tilde{D}_4)$ by Proposition
\ref{prop-aweyl}. But this is not the case. Therefore by
Proposition \ref{prop-d4-free}
the group $\bar{G}$ is $\Symm{n}$ for some $n \ge 3$.
For $a\in D$, if $|[a,N]| > 2$, then within
$a^N$ there are enough elements of $D$ to produce 
together with $b$ a
diagram $\tilde{D}_4$ as in Proposition \ref{prop-aweyl}. But then
$G$ must contain a subgroup $H=\langle H \cap D\rangle$
that is a central
quotient of $\Weyl_2(\tilde{D}_4)$, against hypothesis. Therefore
$|[a,N]| = 2$. Let $\{a,d\}= a^N = D \cap aN$.
A generating set for $\bar{G}$ containing $\bar{a}$
and having diagram $A_{n-1}$ can then be lifted to an $(n-1)$-subset
$A$ of $G$ with the same diagram and $H=\langle A \rangle \cong
\Symm{n}$. The set $\{d,A\}$ then has diagram $D_n$
and generates $G=\langle D \rangle$ since $D=a^H \cup d^H$. 
Therefore
$G$ is a central quotient of $\Weyl(D_{n})$
by Proposition \ref{prop-weyl}.

Now we may assume $[G,N]=1$ so that $G$ is an extension
of central $N$ by $\bar{G}$, which is one of the groups of
Theorem \ref{thm-symp-type}. If $\bar{G}$ is $\Symm{n}$ 
for some $n \ge 3$,
then $G=\langle D\rangle=\bar{G} \cong \Symm{n} \cong \Weyl(A_{n-1})$
by Lemma \ref{lem-center} and Proposition \ref{prop-weyl}.

If $\bar{G}$ is
$O^\gre_{2m}(2)$, or $Sp_{2m}(2)$ with $m \ge 4$ and
$(m,\gre) \neq (4,+)$, then $\bar{G}$ has a $\bar{D}$-subgroup
$O^-_{8}(2)$. By Lemma \ref{lem-o8} the groups $\bar{G}$
and $G$ are not of $ADE$-type, against hypothesis.
We are left with three possible examples: 
\[
\bar{G} \in \{O_6^-(2),Sp_6(2),O_8^+(2)\}\,.
\]
Thus by Proposition \ref{prop-weyl}, the group 
$G$ is a central quotient of $\Weyl(E_n)$
for $n\in\{6,7,8\}$. Each of these is a genuine
example by Lemma \ref{lem-o8}, the groups $\Weyl(E_6)$
and $\Weyl(E_7)$ being subgroups of $\Weyl(E_8)$
generated by reflections.
\end{proof}

We conclude this section with a lemma that
will enable us to apply Theorem \ref{thm-ade}
in our primitive axial algebras setup.

\begin{lemma}
\label{lem-d42}
\label{lem-lem-5.12}
Let $G$ be a group generated by the normal
set $D$ of $3$-trans\-posi\-tions of symplectic type. 
\begin{enumerate}
\item 
A subgroup $H=\langle H \cap D\rangle$ is generated
by a subset $Y\subseteq D$ with diagram the complete bipartite graph $K_{3,2}$ if and only if
$H$ is a central quotient of $\Weyl(D_4)$ or $\Weyl_2(\tilde{D}_4)$.
\item 
For the subgroup $H=\langle Y \rangle$ of the previous
part, the following are equivalent:
\begin{enumerate}
\item 
some $4$-subset of $Y$ generates a subgroup
isomorphic to $\Symm{4}$;
\item $H$ is generated by a $4$-subset of $Y$ with
diagram $D_4$;
\item $H$ is a central quotient of $\Weyl(D_4)$.
\end{enumerate}
\end{enumerate}
\end{lemma}

\begin{proof}
Let $H=\langle H \cap D\rangle$
be a central quotient of $\Weyl({D}_4)$ or
$\Weyl_2(\tilde{D}_4)$. We show that $H$ is generated by
a subset $Y\subset D$ with diagram $K_{3,2}$.  
By Proposition \ref{prop-klein}(1) we can choose
$\{a,b,c,d\}\subset H\cap D$ 
having diagram $D_4=K_{3,1},$
with parts $\{a,b,c\}$ and $\{d\}$. Further, $K=\langle
a,b,c,d\rangle$ is a central quotient of $\Weyl(D_4)$
and $|d^{O_2(K)}|=4$. In the case $H=K$,
choose any $e \in d^{O_2(H)}=d^{O_2(K)}$ with $e \neq d$.
Then $H=\langle a,b,c,d\rangle=\langle a,b,c,d,e\rangle$ and 
$\{a,b,c,d,e\}$ has diagram $K_{3,2}$ with parts $\{a,b,c\}$ and $\{d,e\},$
by Proposition \ref{prop-klein}(1).

Suppose next that $H$ is a central quotient of
$\Weyl_2(\tilde{D}_4)$. Then $|d^{O_2(H)}|=8$.
Choose $e \in d^{O_2(H)}\backslash d^{O_2(K)}$.
Then $Y=\{a,b,c,d,e\}$ has diagram $K_{3,2}$.
Here $\{ d,a,e \}$ and $\{ a,e,a^{de}\}$ both have
diagram $A_3$ with $\langle d,a,e \rangle=
\langle a,e,a^{de}=d^{ae}\rangle$, a copy of $S_4$. Thus
$W=\{a,b,c,a^{de},e\}$ has diagram $K_{4,1}=\tilde{D}_4$
and $\langle Y \rangle = \langle W \rangle$.
By Proposition \ref{prop-aweyl}, the full group
$H$ is $\langle Y \rangle = \langle W \rangle$
as $e \notin KZ(H)$. This gives one direction of (1).

For the remainder
of the proof of (1) and the proof of (2),
let the subset $Y$ of $D$ have diagram $K_{3,2}$
and generate $H$.
Specifically, let $Y=X \cup Z$ 
with $X=\{a,b,c\}$ and $Z=\{d,e\}$
such that $|xz|=3$ for all $x \in X$ and $z\in Z$
and $(wy)^2=1$ for $w,y \in X$ or $w,y \in Z$.
Let
\[
K=\langle a,b,c,d\rangle.
\]
By Proposition \ref{prop-weyl}(2), $K$ is a central
quotient of $\Weyl(D_4)$.

\medskip
(1)
If $e \in K$, then $K=H$ is a central
quotient of $\Weyl(D_4)$, as claimed. So we may assume
$e \notin K$.
As such, the result follows directly by
checking the list of \cite[Theorem 6.6]{HSo}; but we provide a 
direct proof here.

By Proposition \ref{prop-klein}
the set 
$D \cap K$
consists of $12$ transpositions with
diagram the complete tripartite graph $K_{4,4,4}$, and every 
$\Symm{3}$ subgroup $S=\langle 
D \cap S\rangle \le K$ meets each of the parts exactly once.
As $H$ is symplectic,
$C_K(e)$ 
must meet each such $S$ in at least
one element of $D\cap S$. 
The only proper subgroups 
$J=\langle J \cap D\rangle$ of $K$
with this property are those isomorphic to $\Symm{4}$ and
the three elementary abelian $2$-groups generated by one
of the parts.  Again by Proposition \ref{prop-klein} any $J\,(\le K)$ 
isomorphic to $\Symm{4}$ {would contain
three pairs of commuting $3$-transpositions, and so}
at least two {members} of $\{a,b,c,x\}$, the part of 
$D \cap K$
containing $\{a,b,c\}$. But then $J=C_K(e)$ 
would contain at least one of $\{a,b,c\},$ 
which is not the case by hypothesis.

Therefore 
$D \cap C_K(e) =\{d,f,g,h\}$, the part of $D \cap K$
that contains $d$. Here 
$\{a,d,f,g\}$ has diagram $D_4$ and so
$\langle a,d,f,g \rangle =K$ 
and $H = \langle K,e \rangle = \langle
a,d,f,g,e\rangle$ with $\{a,d,f,g,e\}$ having
diagram $\tilde{D}_4=K_{4,1}$. 
As $K$  is a central quotient of $\Weyl(D_4)$ and $e \notin K$,
the group $H$ is a central quotient of $\Weyl_2(\tilde{D}_4)$ 
by Proposition \ref{prop-aweyl}.

\medskip

(2) $(a) \implies (b)$:
A $4$-subset of $Y$ not containing both $d$ and $e$
generates an abelian group or a central quotient
of $\Weyl(D_4)$ by Proposition \ref{prop-weyl}.
Suppose instead that
$\langle a,b,d,e\rangle$ is a copy of $\Symm{4}$.
Then $\langle a,b,d,e\rangle=\langle a,b,d\rangle$,
so $H=\langle a,b,c,d\rangle=K$.

$(b) \implies (c)$: This follows from Proposition \ref{prop-weyl}(2).

$(c) \implies (a)$: Let $E=a^{O_2(H)}=\{a,b,c,x\}$.
For each $y \in E$ the set
$\{y,d,e\}$ has diagram $A_3=D_3=K_{1,2}$.
Therefore $\{d,e\}$ is in exactly two subgroups $K_1$
and $K_2$ of $H$ isomorphic to $\Symm{4}\,(\cong\Weyl(A_3))$. For these
we have $|K_i \cap E|=2$ and $|K_1\cap K_2 \cap E|=0$.
Therefore there is an $i\in\{1,2\}$ with $|K_i \cap \{a,b,c\}|=2$,
hence $|K_i \cap \{a,b,c,d,e\}|=4$ and $K_i=
\langle K_i \cap \{a,b,c,d,e\}\rangle$ is a copy of $\Symm{4}$.
\end{proof}

\section{Jordan algebras of Clifford type}\label{sect jordan}

In this section we discuss a class of Jordan algebras that appear as
subalgebras of the Jordan algebra $\Cl(V,q)^+$, which comes from the
Clifford algebra  $\Cl(V,q)$ of the quadratic space $(V,q)$.
These appear in \cite[Example (3.5)]{HRS} where they are
denoted $V^J(b)$ (for $b$ equal to  half the form $B$ defined below).
In \cite[3.6, p.~74]{Mc} these algebras are called {\it Jordan spin factors}
and are denoted $\mathcal{J}\mathcal{S}pin(V,B)$.

We also prove
a result connecting primitive axial algebras of Jordan type $\half$
and these Jordan algebras of Clifford type (see Theorem \ref{thm jvb}).  
This result will be used in \S 6. 

As is well known,  if $M$ is an associative 
algebra over of field $\ff$ of characteristic not two then the same $M$ taken 
with the product $x\ast y=\half(xy+yx)$ is a Jordan 
algebra. This Jordan algebra is denoted $M^+$.

Let $V$ be a vector space over $\ff$ endowed with a quadratic form $q$. 
Let $B(u,v)=\half(q(u+v)-q(u)-q(v))$ be the associated symmetric 
bilinear form (and so $q(u)=B(u,u)$).

Consider the Clifford algebra $\Cl(V,q)$. This is an associative 
unital  algebra (having the identity $\e$) which is generated by $V$ and satisfies the relations $u^2=q(u)\e, u\in V$. 
Equivalently 
we have relations $uv+vu=2B(u,v)\e,$ for all $u,v\in V$. 
Thus $\Cl(V,q)$ with the product $x\ast y=\half(xy+yx)=B(u,v)\e$
is a Jordan algebra.

It is easy 
to see that $\e\ast \e=\e$ and $\e\ast u=u$ for $u\in V$. Therefore, the 
subspace $\ff \e\oplus V$ of $\Cl(V,q)$  is a subalgebra of the 
Jordan algebra $\Cl(V,q)^+$, hence itself a Jordan algebra. 
We say that this Jordan algebra is of {\it Clifford type} and
denote it by $J(V,B)$.

Here are some relevant properties of $J(V,B)$. (Many of these can be
found in \cite{HRS}, sometimes with different notation.)
Recall the notion of a Miyamoto involution from Notation \ref{def miyamoto},
and the Notation in \ref{not axes}(2).
 
\begin{lemma}\label{lem jvb}
Let $J=J(V,B)$.
\begin{enumerate}
\item 
For $u\in V$ and $\ga\in \ff$, the vector $a:=\ga \e+u$ is an idempotent 
if and only if $(i)$ $a\in\{0,\e\}$ or $(ii)$ $\ga=\half$ and $q(u)=\frac{1}{4}$.

\item
Assume that $a=\half \e+u$ is an idempotent in $J$.  Then  
$J_1(a)=\ff a$ (and so $a$ is a $\half$-axis), $J_0(a)=\ff(\half \e-u)$, 
and $J_{\half}(a)=u^\perp=\{v\in V\mid B(u,v)=0\}$.

\item
For $a$ as in (2), $J$ decomposes into a directs sum $J_{+}\oplus J_{-},$
where $J_{+}=J_1(a)\oplus J_0(a)$ and $J_{-}=J_{\half}(a),$
with $J_{\delta}J_{\epsilon}=J_{\delta\epsilon}$.  
The Miyamoto involution $\tau(a)$ fixes $\e$ and acts on $V$ as 
minus the reflection through $u^{\perp}$ $($that is $v^{\tau(a)}=-v,$ for $v\in u^{\perp}$
and $u^{\tau(a)}=u)$. $($We recall that Jordan theorists call $\gt(a)$ the Peirce reflection of $a.)$ 
\end{enumerate}
\end{lemma}

\begin{proof}
We have that $(\ga \e+u)\ast(\ga \e+u)=\ga^2 \e+\ga \e\ast u+u\ast\ga \e+u\ast u=
(\ga^2+q(u)) \e+2\ga u$. Hence $\ga \e+u$ is an idempotent if and only if $2\ga u=u$, 
and $\ga^2+q(u)=\ga$.  This shows (1).

Now $a\in J_1(a)$ because $a^2=a,$ and then $\e-a=\half\e-u\in J_0(a)$.  
Next, $a\ast v=(\frac{1}{2}\e+u)\ast v=\frac{1}{2}v+B(u,v)=\frac{1}{2}v,$
for $v\in u^{\perp},$ so  
$u^{\perp}\subseteq J_{\frac{1}{2}}(a)$. Since $a$, $\e-a$, and $u^\perp$ 
together span all of $J$, (2) holds.
Part (3) is immediate from  (2).
\end{proof}

\medskip
%
\begin{remark}
Let $J(V,B)=\ff\e\oplus V$ be a Jordan algebra of Clifford type.
Since $v\ast w=B(v,w)\e$ for all $v,w\in V,$ Lemma \ref{lem jvb}
implies  that 
\begin{quote}
{\it $J(V,B)$ is a 
primitive axial algebra of type $\frac{1}{2}$ if and only if $V$ is 
linearly spanned by vectors $u\in V$ with $q(u)=\frac{1}{4}$.}
\end{quote}
In this case it will be appropriate to refer to $J(V,B)$ as an
{\it axial algebra of Clifford type}.
Then the  $\half$-axes of $J$ have the form $\half\e+u,$
with $q(u)=\frac{1}{4}$.  Furthermore $\half\e-u$
is also an absolutely primitive idempotent, as $q(-u)=\frac{1}{4}$.
Finally Lemma \ref{lem jvb} implies that $\tau(a)=\tau(b),$
where $a=\half\e+u$ and $b=\e-a=\half\e-u$ are distinct
axes.
\end{remark}

Next we prove a result that enables us
to identify a primitive axial algebra of Jordan type
$\half$ as a Jordan algebra of Clifford type.
Throughout the rest of this section $A$ is a primitive axial
algebra of type $\half$ generated by a set of  $\half$-axes 
$\cala$.

\begin{lemma}\label{lem quadratic}
Assume that $A$ contains an identity element $\e$.
For $a\in\cala$ set $v_a:=a-\half\e$.  Suppose further that
\[\tag{$*$}
v_av_b\in\ff\e\quad\text{for all }a,b\in\cala. 
\]
Then  $A=J(V,B)$ for some vector 
space $V$ and a symmetric bilinear form $B$ on $V$.
\end{lemma}
 \begin{proof}
 If $A= \ff\e$ then the claim holds with $V=0$. Let us assume that 
$A\neq\ff\e$. In particular, $\e$ is not an axis.

We set $V$ to be the $\ff$-linear span of $v_a$ for all $\half$-axes 
$a\in \cala$. It follows from $(*)$ that $uv\in\ff\e$ 
for all $u,v\in V$.

Note that $V+\ff\e$ is closed for multiplication and 
contains $\cala$. Hence $A=V+\ff\e$. Let $a\in\cala$. Note that 
if $\e\in V$ then $a-\frac{1}{2}\e=(a-\frac{1}{2}\e)\e=
v_a\e\in\ff\e$. This yields that $a\in\ff\e$, 
and so $a=\e$, a contradiction. Therefore, $A=V\oplus\ff\e$. 

Let us define the bilinear form $B$ on $V$ by $uv=B(u,v)\e$. 
Clearly, $B$ is symmetric since $A$ is commutative. Also $B$ is 
bilinear, since the algebra product is bilinear. 
Hence, by definition, $A=J(V,B)$.
\end{proof}

\begin{thm}\label{thm jvb}
Assume that $A$ contains two  $\half$-axes $a, b\in\cala$
such that $a+b=\e$ is the identity element of $A$ and
such that $v_av_c\in\ff\e,$ for all $c\in\cala,$ where $v_c=c-\half\e$.
Then  $A=J(V,B)$ for some vector 
space $V$ and a symmetric bilinear form $B$ on $V$.
\end{thm}
\begin{proof}
We show that $(*)$ of Lemma \ref{lem quadratic} holds.
Let $c,d\in\cala$.
Note that 
\[\textstyle{
v_b=b-\half\e=(\e-a)-\half\e=\half\e-a=-v_a.}
\]
Hence also $v_bv_c\in\ff\e,$ for all $c\in\cala$.
Also (recall the notation $\gs_{c,d}$ from Notation \ref{not alg}(2)),
%
\begin{equation}\label{eq vv}
\textstyle{
v_cv_d=(c-\half\e)(d-\half\e)=
cd-\half c-\half d+\frac{1}{4}\e=\gs_{c,d}+\frac{1}{4}\e.}
\end{equation}
Set $\gs=\gs_{c,d}$.   We show that
%
\begin{equation}\label{eq vcvd}
v_cv_d\in\ff\e.
\end{equation} 
If $c=a$ or $b$ then equation $\eqref{eq vcvd}$ holds by hypothesis.
So assume now that $\{c,d\}$ is disjoint from 
$\{a,b\}$. Notice that
\[\textstyle{
v_av_a=
(a-\frac{1}{2}\e)(a-\frac{1}{2}\e)=a^2-\frac{1}{2}a-\frac{1}{2}a
+\frac{1}{4}\e=\frac{1}{4}\e.}
\]
Let $\gl$ be defined by $v_av_c=\gl \e$. Then $v_c-4\gl v_a$ is a 
$\half$-eigenvector for ${\rm ad}_a$. Indeed, 
$(a-\half\e)(v_c-4\gl v_a)=v_a(v_c-4\gl v_a)=v_av_c-4\gl v_av_a=
\gl \e-4\gl\frac{1}{4}\e=0$. Similarly, $v_d-4\gd v_a$ is also a 
$\half$-eigenvector for ${\rm ad}_a$, where $\gd$ is defined by 
$v_av_d=\gd \e$.
 
In view of the fusion rules in $A$, $(v_c-4\gl v_a)(v_d-4\gd v_a)$ lies in 
$A_1(a)+A_0(a)$. On the other hand, 
\begin{gather*}
(v_c-4\gl v_a)(v_d-4\gd v_a)=v_cv_d-
4\gd v_cv_a-4\gl v_av_d+16\gl\gd v_av_a\\
=v_cv_d-4\gd\gl \e-4\gl\gd \e+
4\gl\gd \e=v_cv_d-4\gl\gd \e.
\end{gather*}
Since $\e\in A_1(a)+A_0(a)$, we conclude 
that $v_cv_d\in A_1(a)+A_0(a)$. 
Now since $a$ and $b$ are absolutely primitive, $A_1(a)=\ff a$.  Also  $\ff b=A_1(b)=A_0(a),$ since $b=\e-a$.
It follows that $v_cv_d\in \ff a\oplus \ff b$. 
Recall that $v_cv_d=
\gs+\frac{1}{4}\e$. Hence $\gs$ is contained in $\ff a\oplus\ff b$. 
 
We now note that $ab=0$, which means that $\ff a\oplus \ff b$ is 
isomorphic to the associative algebra $\ff\oplus \ff$. This algebra does not have 
nilpotent elements. Since $\gs^2=\pi\gs$, where $\pi=\pi_{c,d}$, 
either $\gs=0$ or $\pi\neq 0$ and $\frac{1}{\pi}\gs$ is an idempotent, 
namely $\frac{1}{\pi}\gs$ is one of  $a$, $b$, or $\e=a+b$. Let us look at these possibilities in 
turn. If $\gs=0$ then $v_cv_d=\frac{1}{4}\e$, a multiple of $\e$. 
If $\frac{1}{\pi}\gs=a$ then $a$ is the identity in the subalgebra 
generated by $c$ and $d$. However, this means that $a$ is not absolutely primitive, 
a contradiction. Symmetrically, we also rule out the possibility that 
$\frac{1}{\pi}\gs=b$. Finally, if $\frac{1}{\pi}\gs=\e$ then 
$v_cv_d=\gs+\frac{1}{4}\e=(\pi+\frac{1}{4})\e$, again a multiple of $\e$. 
Hence equation \eqref{eq vcvd} holds and the proof is complete.
\end{proof}

\section{The graph $\gD$ and some consequences}\label{sect delta}
The purpose of this section is to discuss the graph
$\gD$ given in Notation \ref{not delta}(1) below.
Our main result in this section is Theorem \ref{thm ta=tb}.
Throughout this section $\cala$ is a generating set
of $\eta$-axes of the axial algebra $A$.

\begin{notation}\label{not delta}
\begin{enumerate}
\item
We define the graph $\gD$ as follows.  The
vertex set of this graph is the set $\calx$ of all the $\eta$-axes
in $A$.  Two distinct axes $x,y\in \calx$ form an edge if and only if $xy\ne 0$.

\item
For a subset $\calb\subseteq\calx$ we denote by $\gD_{\calb}$
the full subgraph of $\gD$ on the vertex set $\calb$.

\item
Recall the notation $\calb^1$ and $\calb^{\eta}$ from Notation \ref{not axes}(4).

\item
For a subset $\calb\subseteq\calx$ we denote by $\calb^u$
($u$ for unique) the set $\{x\in\calb^{\eta}\mid \tau(x)\ne\tau(y)\text{ for all }x\ne y\in\calb\}.$

\item
$\calb^{nu}$ ($nu$ for not unique) is the set $\calb^{\eta}\sminus\calb^u$.
\end{enumerate}
\end{notation}

\begin{remark}\label{rem calx1}
By Lemma \ref{lem closure}(6), if $a\in\calx^1,$ then $\{a\}$ is a connected
component of $\gD$ and hence $\{a\}$
is a connected component of $\gD_{\calb}$ for all $a\in\calb\subseteq\calx$.
Hence from now on we may assume that 
\[
\cala=\cala^{\eta}.
\]
\end{remark}

\begin{lemma}\label{lem components1}
Let $\cala_1$ and $\cala_2$ be two distinct
connected components of $\gD_{\cala},$
and let $\{\cala_i\mid i\in I\}$ be the set
of connected components of $\gD_{\cala}$.  Then
\begin{enumerate}
\item
$[G_{\cala_1},\, G_{\cala_2}]=1;$

\item
$b^{g_2}=b,$ for all $b\in[\cala_1]$ and $g_2\in G_{\cala_2};$

\item
$b_1b_2=0,$ for all $b_1\in[\cala_1]$ and $b_2\in[\cala_2];$

\item
the graph $\gD_{[\cala_1]}$ is connected;

\item
$[\cala]=\dot\bigcup_{i\in I}[\cala_i]$ is a disjoint union;

\item
there is a bijection $\cala_i\mapsto [\cala_i]$ between the connected
components $\{\cala_i\mid i\in I\}$ of $\gD_{\cala}$ and the
connected components $\{[\cala_i]\mid i\in I\}$ of $\gD_{[\cala]}$.
\end{enumerate}
\end{lemma}
\begin{proof}
(1):\quad 
By definition $a_1a_2=0,$ for all $a_1\in\cala_1$ and $a_2\in\cala_2$.
By Lemma \ref{lem 2b} we see that $\tau(a_1)$ commutes with $\tau(a_2)$.
Since $G_{\cala_i}=\lan \tau(a_i)\mid a_i\in\cala_i\ran,$ $i=1,2,$ part (1)
follows. 
\medskip

\noindent
(2):\quad
Assume first that  $b\in\cala_1$.  Then $b^{\tau(a_2)}=b,$  for all $a_2\in\cala_2,$
because $ba_2=0$.  It follows that $b^{g_2}=b$.   Next write $b=a_1^{g_1}$ with $a_1\in\cala_1$ and $g_1\in G_{\cala_1}$
(see Lemma \ref{lem closure}(1)).
Then using (1) we get $b^{g_2}=a_1^{g_1g_2}=a_1^{g_2g_1}=a_1^{g_1}=b$.
\medskip

\noindent
(3):\quad
Write $b_i=a_i^{g_i},$ with $a_i\in\cala_i$ and $g_i\in G_{\cala_i},$ $i=1,2$.
Then, by (1) and (2), $b_1b_2=a_1^{g_1g_2}a_2^{g_1g_2}=(a_1a_2)^{g_1g_2}=0$.
\medskip

\noindent
(4):\quad
This follows from Lemma \ref{lem closure}(3). 
\medskip

\noindent
(5):\quad 
We first show that $[\cala]=\bigcup_{i\in I}[\cala_i]$.  Clearly it
suffices to show that $[\cala]$ is contained in the union.
Let $b\in[\cala],$ then, by Lemma \ref{lem closure}(1), there exists $a\in\cala$ and
$g\in G_{\cala}$ such that $b=a^g$. Let $i\in I$ so that $a\in\cala_i$.
Using (1) and (2) it follows that $a^g=a^{g_i}$ for some $g_i\in G_i$.
Hence $b\in[\cala_i]$.  Since by (4), $\gD_{[\cala_i]}$
is connected, for each $i\in I,$ the fact that the union is disjoint is immediate from (3).
\medskip

\noindent
(6):\quad
This follows from (4) and (5).
\end{proof}

\begin{lemma}\label{lem component}
Let $\{\cala_i\mid i\in I\}$ be the set of connected components of $\gD_{\cala}$.
For each $i\in I,$ let $A_i=N_{\cala_i}$ (see Notation \ref{not alg}) and let $\calx_i$ be the set of all 
$\eta$-axes in $A_i$.  Then
\begin{enumerate}
\item
$A_iA_j=\{0\}$ for all $i\ne j,$ so $A$ is the sum of its ideals $\{A_i\mid i\in I\};$

\item
the connected components of $\gD$ are $\{\calx_i\mid i\in I\}$.
\end{enumerate}
\end{lemma}
\begin{proof}
(1):\quad
By Lemma \ref{lem closure}(8), $A_i=N_{[\cala_i]}$ and by
\cite[Corollary 1.2]{HRS}, $A_i$ is spanned over $\ff$ by $[\cala_i]$.
By Lemma \ref{lem components1}(3), $A_iA_j=\{0\}$ for $i\ne j$.
Hence $A_i$ is an ideal of $A$.  Since the sum of $A_i$ contains $\cala$
we see that it equals $A$. 
\medskip

\noindent
(2):\quad
Let $u\in \calx_i$.  Suppose that $ua=0$ for all $a\in[\cala_i]$.
Since $A_i$ is spanned by $[\cala_i],$ there is a basis $\calb_i$
of $A_i$ such that $\calb_i\subseteq [\cala_i]$.  Hence there 
is $b\in\calb_i$ and $0\ne\ga\in\ff$ such that $b=\ga u+x,$
with $x\in{\rm Span}(\calb_i\sminus\{b\})$.  But then
$0=ub=\ga u,$ a contradiction.  It follows that $u$ is in the
same connected component of $[\cala_i]$ of $\gD$ (see Lemma \ref{lem components1}(6)).  By (1),
$\calx_i$ is a connected component of $\gD$.
\end{proof}

\begin{prop}\label{prop ta=tb}
Let $a, b\in\calx^{\eta}$ be two distinct  $\eta$-axes in $A$.
Assume that $\tau(a)=\tau(b)$.  Then $ab=0$
and
\begin{enumerate}
\item
for any  $\eta$-axis  
$c\in A$ exactly one the following holds:
\begin{itemize}
\item[(i)]
$ac=bc=0.$

\item[(ii)]
$\eta=\half,$ and for some $x\in\{a,b\}=\{x,y\},$ we have $N_{x,c}=B(\half,0)_{x,c}$ is $3$-dimensional,
$N_{y,c}=J_{y, c}$ (see Lemma \ref{lem 2dim}(1c)) and $N_{y,c}\subset N_{x,c}$.
Further $N_{x,c}$ contains an identity $1_{x,c}=a+b$.

\item[(iii)]
$\eta=\half,$ 
$N_{a,c}=N_{b,c}$ is $3$-dimensional
and contains an identity $1_{a,c}=a+b$.  
\end{itemize}
\item
If $d$ is an $\eta$-axis in $A$ such that $\tau(d)=\tau(a),$
then $d\in\{a,b\}$.
\end{enumerate}
\end{prop}
\begin{proof}
By Lemma \ref{lem 2b}, $N_{a,b}=2B_{a,b}$.
If $N_{a,c}=2B_{a,c},$ then $c=c^{\tau(a)}=c^{\tau(b)},$
so by Lemma \ref{lem 2b} again, $N_{b,c}=2B_{b,c}$.
Hence part (1i) holds.

So we may assume that both $N_{a,c}$ and $N_{b,c}$ are not of type $2B$.
If $\eta\ne\half,$ then by \cite[Prop.~4.8]{HRS},
$N_{x,c}=B(\eta,\half\eta)_{x,c},$ for $x\in\{a,b\}$.
By Lemma \ref{lem 3ceta}, $a^{\tau(c)}=c^{\tau(a)}=c^{\tau(b)}=b^{\tau(c)},$
and then $a=b,$ a contradiction.

Hence $\eta=\half$.
Let 
\[
V:=N_{c,c^{\tau(a)}}\subseteq N_{a,c}\cap N_{b,c}.
\]
Suppose that $N_{a,c}$ is $2$-dimensional. 
By Lemma \ref{lem 2b}, $c\ne c^{\tau(a)}$ so 
$N_{a,c}=V$.     If $N_{b,c}$ is $2$-dimensional,
then $N_{a,c}=V=N_{b,c}$ is spanned by $a$ and $b$.  So $N_{a,c}=N_{a,b}=N_{b,c}$
is of type $2B,$  a contradiction.
Hence $N_{a,c}=V\subset N_{b,c}$.
Since $a, b\in N_{b,c}$ and $ab=0,$
Lemma \ref{lem ab=0} implies that $1_{b,c}=a+b,$
and the last part of (1ii) holds.  
Also, by Lemma \ref{lem vsubb}, if $N_{b,c}=B(\half,\half)_{b,c},$ then $1_{b,c}\in N_{c, c^{\gt(b)}}=V=N_{a,c}$.
But then $b\in N_{c, c^{\gt(b)}},$ a contradiction.  Hence, by Lemma \ref{lem vsubb}
the first part of  
(1ii) holds.   

If $V$ is $3$-dimensional, then since $V=N_{c, c^{\tau(b)}}$
we see that $N_{a,c}=V=N_{b,c}$ and as above 
$1_{a,c}=a+b$ and the proposition holds.
 
Hence we may assume that $N_{a,c},\ N_{b,c}$ are $3$-dimensional and $V$ is $2$-dimensional.
We use Lemma  \ref{lem vsubb}.  We must consider $2$ cases.
 
Assume first that $V=2B_{c, c^{\tau(a)}}$.  Then $N_{x,c}=B(\half, \half)_{x,c},$
for $x\in\{a,b\}$.   By Lemma \ref{lem halfhalf}(2),
$c^{\tau(x)}=1_{x,c}-c,$ for $x\in\{a,b\}$.  Since $\tau(a)=\tau(b)$
we see that
\[
\e:=1_{a,c}=1_{b,c}.
\]
Also $x^{\tau(c)}=\e-x$ for $x\in\{a,b\}$.  We have $(\e-a)b=b$.  Since $\e-a$ is an axis $A$
this forces $\e-a=b$.  So we see that $N_{a,c}=N_{b,c}$
and $a+b$ is its identity element.

Assume next that $V=J_{c, c^{\tau(a)}}$.
Then, by Lemma \ref{lem vsubb}(i), $N_{a,c}=B(\half,0)_{a,c}$
and $N_{b,c}=B(\half,0)_{b,c}$.
We claim that $A_{\half}(a)=A_{\half}(b)$.  Indeed,
for any $v\in A$ we have $v\in A_{\half}(a)$ if and only if $v^{\tau(a)}=-v$.
and the same holds for $b$.  Since $\tau(a)=\tau(b)$
the claim follows.
 
Now by Theorem \ref{thm betaphi}(5) (since $\gvp_{x,c}=0,$ for $x\in\{a,b\}$),
\[\textstyle{
v_x:=\half  x+\half c+\gs_{x,c}\in (N_{x,c})_{\half}(x),\text{ for }x\in\{a,b\}.}
\]
Hence
\[\textstyle{
v_a-v_b= (\gs_{a,c}+\half a)-(\gs_{b,c}+\half b)\in A_{\half}(a).}
\]
Note now that $\tau(a)=\tau(b)$ fixes $v_a-v_b$.  But it also negates it.
Hence
\[\textstyle{
-\gs_{a,c}-\half a=-\gs_{b,c}-\half b}.
\]
Since $\pi_{x,c}=-\half$ for $x\in\{a,b\},$ we get: $\half 1_{a,c}-\half a=\half 1_{b,c}-\half b,$ or
\[
1_{a,c}-a=1_{b,c}- b.
\]
It follows that $1_{a,c}b=(1_{a,c}-a)b=(1_{b,c}-b)b=0$.  So $1_{a,c}b=0$.
Applying $\tau(c)$ we see that $1_{a,c}b^{\tau(c)}=0$.
But 
\[
\dim\big({\rm Span}\big(\{b, b^{\tau(c)}\}\big)\big)=\dim\big({\rm Span}\big(\{c, c^{\tau(b)}\}\big)\big)=\dim(V)=2,
\]
and both spaces live in the algebra $N_{b,c}$ of dimension $3$.
Hence $W:=V\cap {\rm Span}(\{b, b^{\tau(c)}\})\ne 0$ and $1_{a,c}$ annihilates $W$.
However $W\subset N_{a,c}$ and $1_{a,c}$ is the identity of this algebra,
and we finally reached our contradiction.
This proves (1).

Let $d$ be an $\eta$-axis in $A$ such that $\tau(d)=\tau(a)$
and $d\ne a$.  Since $a\in \calx^{\eta}$ also $d\in\calx^{\eta}$.  Let $c\in\calx$
such that $c^{\tau(a)}\ne c$.  This is possible since $\tau(a)\ne {\rm id}$.
Then $N_{a,c}\ne 2B_{a,c}$.  So either (1ii) or (1iii) hold and we may assume
that $N_{a,c}$ is $3$-dimensional and contains an identity $1_{a,c}=a+b$.
By (1ii) and (1iii) applied to $d$ in place of $b$ we see that 
also $a+d$ is the identity of $N_{a,c}$.
Hence $a+d=1_{a,c}=a+b,$ and $d=b$.
\end{proof}

\begin{prop}\label{prop e}
Let $\gD_{\cala}$ be as in Notation \ref{not delta}(3).  Assume there are distinct $a,b\in\cala^{\eta}$ such that 
$\tau(a)=\tau(b)$.  Then $\eta=\half$ and $a,b$ are contained in a connected component
$\calb$  of $\gD_{\cala}$.
 
Let $B=N_{\calb}$ be the subalgebra of $A$ generated by $\calb,$ and let
$\calc$ be the set of all  $\half$-axes in $B$. 
Then 
\begin{enumerate}
\item
$\calc$ is a connected component of $\gD;$

\item
$xa\ne 0\ne xb,$ for all  $x\in\calc\sminus\{a,b\};$

\item
$B$ contains an identity element $\e=a+b;$

\item
for any $x\in\calc$ such that   
$N_{a,x}$ is $3$-dimensional we have $\e=1_{a,x}$.
\end{enumerate}
\end{prop}
\begin{proof}
Part (1) is Lemma \ref{lem component}(2).  
Since $a,b\in\cala^{\eta},$ there exists $c\in\cala$ such that $c^{\tau(a)}\ne c$
(indeed if $c^{\tau(a)}=c,$ for all $c\in\cala,$ it would follow that
$\tau(a)={\rm id}$ as $A$ is generated by $\cala$).
Hence $ac\ne 0$ and $bc\ne 0$.  Thus $a,b$ are in the same
connected component $\Delta_{\cala}$.

By Proposition \ref{prop ta=tb}, $N_{a,b}=2B_{a,b}$.
Let $d(\ ,\ )$
be the distance function on $\gD$.
Since $\calc$ is connected there exists 
$c\in \calc$ with $d(a,c)=1$. 
Thus $c\ne b,$ and by Proposition \ref{prop ta=tb},
$d(b,c)=1$ and  we may assume without loss
that $N_{a,c}$ is $3$-dimensional and
contains $1_{a,c}=a+b$.  Also $\eta=\half$.  Set
\[
\e=1_{a,c}=a+b.
\]
Consider the set
\[
\calc_1(a):=\{x\in\calc\mid d(a,x)=1\}.
\]
Replacing $c$ with $x\in\calc_1(a)$ in the above argument shows that 
there exists $u\in\{a,b\}$ such that $N_{u,x}$ is
$3$-dimensional and contains an identity $1_{u,x}=a+b=\e$. 
Hence $\e x=x,$ for all $x\in\calc_1(a)$.
Note also that by Proposition \ref{prop ta=tb}, $\calc_1(a)=\calc_1(b)$.

Let $y\in\calc\sminus\calc_1(a)$ be at distance $2$ from $a$ in $\gD$.  
Then $ay=0=by$ and we can find $x\in \calc_1(a)$ such that $d(x,y)=1$.
Thus by Proposition \ref{prop e} (without loss after perhaps interchanging $a$ and $b$), we have
$\e=1_{a,x}$ and $ay=0$.

Notice that by Proposition \ref{prop ta=tb}(1i), $(a+b)^{\tau(y)}=a+b,$
that is $\e^{\tau(y)}=\e$.  Also $N_{x,y}\ne 2B_{x,y}$.

Now $\e y=0$ and hence $\e y^{\tau(x)}=0$ (because $\e=1_{a,x}$ so $\e^{\tau(x)}=\e$).
Also $\e x=x,$ so $\e x^{\tau(y)}=x^{\tau(y)}$ (because $\e^{\tau(y)}=\e$).
Let $W:={\rm Span}\big(\{y,\ y^{\tau(x)}\}\big)\cap {\rm Span}\big(\{x,\ x^{\tau(y)}\}\big)$.
Since $N_{x,y}$ is at most $3$-dimensional, $W\ne 0$.  But the above
shows that multiplication by $\e$ both annihilates $W$ and acts as
the identity map on $W,$ a contradiction.

Hence $\calc_1(a)=\calc\sminus\{a, b\}$ and clearly $d(a,b)=2$
in $\calc$.  But now, as we saw above $\e x=x$ for
all $x\in\calc$.  By Lemma \ref{lem identity} (with $\calb$ in place
of $\cala$ and $N_{\calb}$ in place of $A$) we see that $\e$
is the identity of $B$.
\end{proof}
 
We can now prove the main result of this section

\begin{thm}\label{thm ta=tb}
Assume that $\gD_{\cala}$ is connected and that there are distinct $a,b\in\cala^{\eta}$ such that 
$\tau(a)=\tau(b)$.  Then $A=J(V,B)$ is a Jordan algebra of Clifford type.
\end{thm} 
\begin{proof}
We show that the hypotheses of Theorem \ref{thm jvb} are satisfied.
By Proposition \ref{prop e}(3), $a+b=\e$ is the identity element of $A$.
Let $c\in\cala$.  Then clearly $v_av_c\in\ff\e,$ for $c\in\{a,b\}$.
Otherwise, by Proposition \ref{prop e}(2), $N_{a,c}$ is not $2B_{a,c}$.  
Also, as in equation \eqref{eq vv} (in the proof of Theorem \ref{thm jvb}), 
$v_av_c=\gs_{a,c}+\frac{1}{4}\e$.  If $N_{a,c}$ is $2$-dimensional,
then by Lemma \ref{lem 2dim}, $\gs_{a,c}=0$ (because $ac\ne 0$), and so $v_av_c\in \ff\e$.
If $N_{a,c}$ is $3$-dimensional, then by Proposition \ref{prop e}(4), $\e=1_{a,c},$
so by Theorem \ref{thm betaphi}(3), $\gs_{a,c}=\pi_{a,c}\e$ and again $v_av_c\in\ff\e$.
\end{proof}

\begin{Def}\label{def components}
\begin{enumerate}
\item
Let the ideals $A_i,$ $i\in I$ be as in Lemma \ref{lem component}.
We call $A_i$ the {\it components} of the algebra $A$.
Note that by Lemma \ref{lem component}(2), the components $A_i$
are independent of $\cala$.

\item
Let $A_i$ be a componet of the algebra $A$.  If $A_i$ is a Jordan
algebra of Clifford type we call $A_i$ a component of {\it Clifford type}.  Otherwise
(in the case where $\tau(x)\ne\tau(y)$ for distinct $x,y\in\calx\cap A_i$)
we call $A_i$ a component of {\it unique type}.

\item
Let $\calb\subseteq\calx^{\eta}$ be a closed set of $\eta$-axes ($\calb=[\calb]$).
Suppose that $\calb$ is contained in a connected component $\calx_i$ of $\gD$.
\begin{enumerate}
\item[(i)]
We say that $\calb$ is  of {\it non-unique type} if there exists distinct $a,b\in\calb$
such that $\tau(a)=\tau(b)$ (and then $\eta=\half$ and $A_i$
is a Jordan algebra of Clifford type).  
\item[(ii)]
We say that $\calb$ is of {\it C-unique type} if $A_i$
is a Jordan algebra of Clifford type and the map $b \mapsto\gt(b)$ is bijective
on $\calb$.  
\item[(iii)]
We say that $\calb$ is of {\it NC-unique type} if $A_i$
is a {\bf not} a Jordan algebra of Clifford type.  (And then
the  map $b \mapsto\gt(b)$ is necessarily bijective
on $\calb$.)
\end{enumerate}
\end{enumerate}
\end{Def}

\begin{remark}
Let $A_i$ be a component of $A$ of Clifford type.
Then $A_i\cap A_j=\{0\}$ for any component $A_j$
of $A$ with $A_j\ne A_i$.  Indeed, by Lemma 
\ref{lem component}(1), $A_iA_j=0,$ so since 
$A_i$ contains an identity element
we must have $A_i\cap A_j=\{0\}$.  If $A_i$
and $A_j$ are distinct components of $A$ of unique
type, then it may happen that $A_i\cap A_j\ne \{0\}$.
\end{remark}

We close this section with a theorem that summarizes some
of the results in this section.

\begin{thm}\label{thm a}
Let $\{\cala_i\mid i\in I\}$ be the set of connected components
of $\gD_{\cala}$. For each $i\in I,$ let $A_i=N_{\cala_i}$ and let
$\calx_i=A_i\cap\calx$.  Then
\begin{enumerate}
\item
$A=\sum_{i\in I}A_i$ is the sum of its ideals $A_i;$ 
\item
$A_iA_j=0,$ for distinct $i, j\in I;$
\item
$\{A_i\mid i\in I\}$ are the components of $A$ and   $\{\calx_i\mid i\in I\}$
are the connected components of $\gD.$
\item
Let $i\in I,$ then exactly one of the following holds:
\begin{enumerate}
\item[(i)]
$[\cala_i]$ is of non-unique type, so $A_i$ is a Jordan algebra of Clifford type.

\item[(ii)]
$[\cala_i]$ is of C-unique type.

\item[(iii)]
$[\cala_i]$ is of NC-unique type.
\end{enumerate}
\end{enumerate}
\end{thm}
\begin{proof}
Parts (1), (2) and (3) are Lemma \ref{lem component}.
Part (4) follows from Theorem \ref{thm ta=tb}.
\end{proof}

\section{The case where $A$ is a $3$-transposition algebra}
\label{sect 3trans}
Recall from subsection \ref{sub intro 3-trans} of the introduction
the notion of a $3$-transposition algebra with respect to a generating
set of $\eta$-axes.  The following theorem is taken from \cite{HRS}:

\begin{thm}[Theorem 5.4 in \cite {HRS}]\label{thm hrs}
Assume $\eta\ne\half$. Then $A$ is a $3$-transposition
algebra with respect to any subset $\calb\subseteq \calx$ that
generates $A$.
\end{thm}
\begin{proof}
It suffices to show that $|\tau(a)\tau(b)|\in\{2,3\}$ for any $a, b\in\calx,$
with $a\ne b$ (note that by Proposition \ref{prop e}, $\tau(a)\ne\tau(b)$).  
If $N_{a,b}$ is $2$-dimensional then by Lemma \ref{lem 2dim},
$N_{a,b}$ is either $3C^{\times}(-1)$ (and $\eta=-1$ so $\charc(\ff)\ne 3$)
and so by Lemma \ref{lem 3c(-1)times}, $|\tau(a)\tau(b)|=3,$ or $N_{a,b}=2B_{a,b}$
and $|\tau(a)\tau(b)|=2$.  If $N$ is $3$-dimensional
then, by \cite[Proposition 4.8]{HRS}, $N=B(\eta,\half\eta)_{a,b},$
so by Lemma \ref{lem 3ceta}, $|\tau(a)\tau(b)|=3$.
\end{proof}

What then about the case $\eta=\half$\,?

Most of this section is devoted to the case where $A$ is a $3$-transposition algebra
with respect to a generating set of $\eta$-axes $\cala\subseteq\calx$.
By Theorem \ref{thm hrs} we may (and we will) assume that
\[\textstyle{
\eta=\half.}
\]
Thus by an axis in $A$ we mean a $\half$-axis.
Note that if $\charc(\ff)=3$ then $\half=-1$.

The Miyamoto involution set $D:=D_{[\cala]}$ is a normal subset of 
$3$-trans\-pos\-i\-tions
generating the group $G:=G_{[\cala]}$.    
Further,
we assume that
$D$ is a conjugacy class in $G$.  This implies that the graph $\gD_{\cala}$
is connected (see Notation \ref{not delta}(2)).  In particular $A$ is either
of Clifford type or of unique type (see Definition \ref{def components}(2)).

Our main result in this section is:

\begin{thm}\label{thm jordan ade}
Assume that $\charc(\ff)\ne 3,$ that $A$ is of Clifford type,
and that $A$ is a $3$-transposition algebra with respect to $\cala$. 
Then $G_{[\cala]}$ is a $3$-transposition group of $ADE$-type. $($See Definition \ref{def-ade}.$)$
\end{thm}

By Remark \ref{rem calx1} we may ignore the axes in $\cala\cap\calx^1$.

\begin{lemma}\label{lem triangle}
Let $G$ be a $3$-transposition group generated
by a conjugacy class of $3$-transpositions $D$.
Let $r, s, t\in D$ be three distinct involutions such that $|uv|=3,$ for all distinct $u, v\in\{r, s, t\}$.
Set $H=\lan r,s,t\ran,$ then
\begin{enumerate}
\item
if $r^s=t,$ then $H\cong S_3;$

\item
if $|r^st|=3,$ then $H\cong  3^2:2$ or $3^{1+2}:2;$ 

\item if $|r^st|=2,$ then $H\cong S_4$.
\end{enumerate}
\end{lemma}
\begin{proof}
See, e.g., \cite[4.1, p.~2526]{HSo}.
\end{proof}
 
\begin{lemma}\label{lem tatb2}
Let $a,b$ be two axes in $A$ and assume that  $|\tau(a)\tau(b)|=2$.
Set $N=N_{a,b}$. Then either
\begin{enumerate}
\item[(i)]  
$A$ is of unique type and 
$N=2B_{a,b}$.   

\item[(ii)]
$A$ is of Clifford type and $N=B(\half,\half)_{a,b}.$
\end{enumerate}
\end{lemma}
\begin{proof}
This is Lemma \ref{lem tatbk}(2).  
Note that if $N=B(\half,\half)_{a,b},$ then 
we must have $\tau(1_{a,b}-x)=\tau(x),$ for $x\in\{a,b\},$
so $A$ is of Clifford type.  Note further that if $A$ is of Clifford type
and $ab=0,$ then $(\e-a)b=b$.  But since $\e-a$ is an
axis in $A$ we must have $b=\e-a,$ so $\gt(a)=\gt(b)$. 
\end{proof}

\begin{lemma}\label{lem tatb3}
Let $a, b$ be two distinct axes in $A$ and set $N=N_{a,b}$.
Then 
\[\tag{$*$}
|\tau(a)\tau(b)|=3,
\]
if and only if one of the
following holds
\begin{itemize}
\item[(i)]
$\dim(N)=2$ and $N=3C(-1)^{\times}_{a,b}$ $($so $\charc(\ff)=3).$

\item[(ii)]
$\dim(N)=3,$  $N$ contains no identity element and $N=3C(-1)_{a,b}$  $($so $\charc(\ff)=3).$

\item[(iii)]
$\dim(N)=3,$ $N$ contains an identity element  $1_{a,b}$
and $N=3C(\half)_{a,b}=B(\half,\frac{1}{4})$ $($so $\charc(\ff)\ne 3)$.  We then have
\[\textstyle{
ab=-\frac{3}{8}1_{a,b}+\half a+\half b.}
\]
\item[(iv)]
$A$ is of Clifford type and $N_{a,b}=B(\half,\frac{3}{4})$
$($so when $\charc(\ff)=3,$ $N_{a,b}=B(-1,0))$.
We then have
\[\textstyle{
ab=-\frac{1}{8}1_{a,b}+\half a+\half b.}
\] 
Further, if $\charc(\ff)\ne 3,$
then  $N=3C(\half)_{x,\ (1_{a,b}-y)},$ while if $\charc(\ff)=3,$
then $N_{x,\, (1_{a,b}-y)}=3C(-1)^{\times}_{x, (1_{a,b}-y)},$ for $\{x,y\}=\{a,b\}.$ 
\end{itemize}
\end{lemma}
\begin{proof}
Suppose that $(*)$ holds. Then by Lemma \ref{lem tatbk} with $k=3$ we either have:
\begin{enumerate}
\item
$a^{\tau(b)}=b^{\tau(a)},$ or 

\item
$\dim(N)=3,$ $N$
contains an identity element $1_{a,b},$ 
$1_{a,b}-x$ is an axis in $A,$ 
$\tau(1_{a,b}-x)=\tau(x)$  
and $1_{a,b}-x=y^{\tau(x)\tau(y)},$ for $\{x,y\}=\{a,b\}$.
\end{enumerate}
Suppose that (1) holds.  Then (i), (ii) or (iii) hold by Corollary \ref{cor halfeta},
and Lemma \ref{lem 3chalfab}.  Also
use Lemma \ref{lem 3ceta}(2) if $\charc(\ff)=3$ (because $\half=-1$ if $\charc(\ff)=3$).

Suppose that (2) holds.  Then (iv) holds by Lemma \ref{lem 3chalfa1-b}.
Finally if (i) holds then $(*)$ holds by Lemma \ref{lem 3c(-1)times},
if (ii) or (iii) hold then $(*)$ holds by Lemma \ref{lem 3ceta} and Lemma \ref{lem 3chalfab}.
Finally, if (iv) holds, then $(*)$ holds by Lemma
\ref{lem 3chalfa1-b}.  
\end{proof}

\begin{cor}\label{cor tatb3}
Let $a,b\in A$ be two distinct axes and suppose 
$|\tau(a)\tau(b)| =3$.   Then either
\begin{enumerate}
\item[(i)]
$A$ is of unique type and $N_{a,b}=3C(\half)_{a,b}$.

\item[(ii)]
$\charc(\ff)=3$ and $N_{a,b}=3C(-1)^{\times}_{a,b}$. 

\item[(iii)]
$A$ is of Clifford type, $\charc(\ff)\ne 3$ and $N_{a,b}=3C(\half)_{a,b}$. 

\item[(iv)]
$A$ is of Clifford type and $N_{a,b}=B(\half,\frac{3}{4})_{a,b}.$
\end{enumerate}
\end{cor}
\begin{proof}
Assume that $A$ is of Clifford type.
Let $\e$ be the identity element $A$.
Since $\tau(a)=\tau(\e-a),$ 
Proposition \ref{prop e} implies that we may assume (after
perhaps interchanging $a$ and $\e-a$) that  
$\e$ is the identity element of $N_{a,b}$.  
Since $3C(-1)_{x,y}$ does not contain an identity element we must have 
$\charc(\ff)\ne 3$ in (iii).  The rest of the lemma 
follows from Lemma \ref{lem tatb3}.
\end{proof}
 
\begin{Def}\label{def 1/8}
Assume $\eta=\half$. Let $a,b\in\calx$.
If $ab=-\frac{3}{8}1_{a,b}+\half a+\half b$
we will say that
$N_{a,b}$ is of type $-\frac{3}{8}$ 
while if  $ab=-\frac{1}{8}1_{a,b}+\half a+\half b$
we will say that
$N_{a,b}$ is of type $-\frac{1}{8}$.  Notice that
if $|\tau(a)\tau(b)|=3$ and
$A$ is of Clifford type, 
then, by Corollary \ref{cor tatb3}, $N_{a,b}$
is necessarily of type $-\frac{3}{8}$ or of type $-\frac{1}{8}$.
(Indeed when $\charc(\ff)=3$ and $N_{a,b}$ is of type $-\frac{3}{8},$
then $N_{a,b}=3C(-1)^{\times}_{a,b}$.)
\end{Def}

\begin{lemma}\label{lem a/8}
Assume that $A$ is of Clifford type, and let $a,b,c\in\calx$ be
three distinct axes.  Then  
we have
\begin{enumerate}
\item
if $N_{x,y}$ is of type $-\frac{3}{8}$ for all distinct $x,y\in\{a,b,c\}$
and $\lan\tau(a),\tau(b), \tau(c)\ran$ is {\em not} isomorphic to $S_3,$ then $N_{a,b^{\tau(c)}}=J_{a,b^{\tau(c)}}$.
In particular, if $\charc(\ff)\ne 3,$ then  $|\tau(a)\tau(b^{\tau(c)})|\notin\{2, 3\};$

\item
if $N_{a,b}$ and $N_{a,c}$ are  of type $-\frac{3}{8}$ and $N_{b,c}$ is
of type  $-\frac{1}{8},$ then  $N_{a,b^{\tau(c)}}=B(\half,\half)_{a, b^{\tau(c)}}$.
In particular, $|\tau(a)\tau(b^{\tau(c)})|=2$.
\end{enumerate}
\end{lemma}
\begin{proof}
By Lemma \ref{lem ta} we have
\[
\pi_{a, b^{\tau(c)}}=8\pi_{a,c}\pi_{b,c}+2\pi_{a,c}-\pi_{a,b}+2\pi_{b,c}.
\]
\smallskip

\noindent
(1):\quad In this case $\pi_{x,y}=-\frac{3}{8},$ for all distinct $x,y\in\{a,b,c\},$
so $\pi_{a, b^{\tau(c)}}=0$.  Hence, by Lemma \ref{lem ta}(2), the first part of (1) holds.  The second
part follows from Lemma \ref{lem 2dim3}.
\smallskip

\noindent
(2):\quad In this case $\pi_{a,b}=\pi_{a,c}=-\frac{3}{8}$ and $\pi_{b,c}=-\frac{1}{8}$
so $\pi_{a, b^{\tau(c)}}=-\frac{1}{4}$.  Now (2) follows from Remark \ref{rem halfhalf}.
\end{proof}
 
\begin{lemma}\label{lem nx}
Suppose that $\charc(\ff)\ne 3,$ that $A$ is of Clifford type, and that 
$A$ is a $3$-transposition algebra with respect to $\cala$.  
Let $\e$ be the identity element of $A$. Then
\begin{enumerate}
\item
Suppose that $\calb=\{a,b,c\}\subset [\cala]$ is a set of size $3$ such that\linebreak
$\lan \tau(a),\tau(b),\tau(c)\ran$ is {\em not} isomorphic to $S_3$   and $|\tau(x)\tau(y)|=3,$ for all distinct $x,y\in\calb$.
Then either $N_{x,y}$ is of type $-\frac{1}{8}$ for all distinct
$x,y\in\calb,$ or there exists distinct $x,y\in\calb$ such that
$N_{x,y}$ is of type $-\frac{1}{8}$ and both $N_{x,z}$ and  $N_{y,z}$ are
type $-\frac{3}{8},$ where $\{x,y,z\}=\{a,b,c\}.$ 

\item
Suppose that $\calb=\{a,b,c,d\}\subset [\cala]$ is such that
\begin{enumerate}
\item
$\big\lan\tau(x)\mid x\in\{a,b,c,d\}\big\ran$ is {\em not} isomorphic to $S_4.$
\item
$|\tau(a)\tau(c)|=2=|\tau(b)\tau(d)|.$

\item
$|\tau(x)\tau(b)|=3=|\tau(x)\tau(d)|,$ for $x\in\{a,c\}.$   
\end{enumerate}
(Thus $\calb$ has diagram $\tilde{A}_3$ (see the definition of a diagram in \S4)).
Then the following {\em does not hold}:
\begin{itemize}
\item[$(*)$]
both $N_{a,b},\ N_{a,d}$ have the same type,
and both $N_{c,b},\ N_{c,d}$ have the same type.
\end{itemize} 
\end{enumerate}
\end{lemma}
\begin{proof}
(1):\quad
Assume (1) does not hold. 
If $N_{x,y}$ is of type $-\frac{3}{8}$ for all distinct 
$x,y\in\{a,b,c\},$  then Lemma \ref{lem a/8}(1)
applies and so $|\tau(a)\tau(b^{\tau(c)})|\notin\{2,3\},$ a contradiction.

Otherwise, for some distinct $x,y\in\calb$ we have $N_{x,y}$ is of type $-\frac{3}{8}$
and both $N_{x,z}$ and $N_{y,z}$ are of type $-\frac{1}{8}$, where $\{x,y,z\}=\{a,b,c\}$.
But by Lemma \ref{lem tatb3}(iv) we see that both $N_{x,\e-z}$ and $N_{y,\e-z}$
are of type $-\frac{3}{8}$.  By Lemma \ref{lem a/8}, $N_{(\e-z), x^{\tau(y)}}=J_{(\e-z), x^{\tau(y)}}$.
Since $\gt(z)=\gt(\e-z),$ we get from Lemma \ref{lem 2dim3} that 
$|\tau(z)\tau(x^{\tau(y)})|\notin\{2,3\},$ a contradiction.
\medskip

\noindent
(2):\quad
Suppose there exists $\calb\subset [\cala]$ satisfying (a), (b) and (c) of   (2),
and $(*)$ holds.  Notice that by hypothesis (b), and by Lemma \ref{lem tatb2},
$N_{a,c}=B(\half,\half)_{a,c}$ so $\e-x\in\cala,$ for
all $x\in\{a,b,c,d\}$.  Interchanging 
$a$ with $\e-a$ and $c$ with $\e-c$ if necessary, we may assume that
$N_{a,b},\ N_{a,d},\ N_{c,b},\ N_{c,d}$ are all of type $-\frac{3}{8}$.
Note that $\lan \tau(x), \tau(y), \tau(z)\ran\cong S_4,$ for all distinct
$x,y,z\in \{a,b,c,d\}$.  Hence $|\tau(x)\tau(c)^{\tau(d)}|=3,$ for $x\in \{a,b \}$.
Notice that $N_{c,c^{\tau(d)}}$ is of type $-\frac{3}{8}$ (because $N_{c,d}=3C(\half)_{c,d}$).
Hence applying (1) to with $\{b,c,c^{\tau(d)}\}$ in place of
$\{a,b,c\}$ and using Corollary \ref{cor tatb3} shows that $N_{a,c^{\tau(d)}}$ 
is of type $-\frac{1}{8}$.  Similarly, $N_{b,c^{\tau(d)}}$
is of type $-\frac{1}{8}$.  By hypothesis (a), $\lan\tau(a),\tau(b),\tau(c^d)\ran$
is not isomorphic to $S_3$.  Hence we apply part (1)
with $c^{\tau(d)}$ in place of $c$ to get  a contradiction.  
\end{proof}
 
Observe that

\begin{thm}\label{thm bijection}
Assume that $\charc(\ff)\ne 3,$ that $A$ is of Clifford type,
and that $A$ is a $3$-transposition algebra with respect to $\cala$. 
Then $G_{[\cala]}$ is a $3$-transposition group of symplectic type.
\end{thm}
\begin{proof}
Let $\e$ be the identity element of $A$. 
Assume that $G_{[\cala]}$ is not of symplectic type. Then there are  $a, b, c\in [\cala]$ 
such that $\lan\tau(a),\tau(b),\tau(c)\ran$ is {\it not} isomorphic
to $S_3,$ and such that   $|\tau(x)\tau(y)|=3$ for all $x,y\in\{a, b, c, b^{\tau(c)}\}$.
By Lemma \ref{lem a/8}(1) there are distinct $x,y\in\{a,b,c\}$ such that $N_{x,y}$
is of type $-\frac{1}{8}$.  Hence by Lemma \ref{lem 3chalfa1-b}, $\e-x\in [\cala],$
so $\e-z\in\cala,$ for all $z\in\{a,b,c\}$.
   
By Lemma \ref{lem 3chalfa1-b},
after perhaps interchanging $x$ with $\e-x$ for $x\in\{b,c\}$
we may assume that 
\[\textstyle{
N_{a,x}=3C(\half)_{a,x}\text{ for }x\in\{b,c\}.}
\]
Thus $\pi_{a,b}=\pi_{a,c}=-\frac{3}{8}$.  
By  Lemma \ref{lem a/8} we have  
$|\tau(a)\tau(b^{\tau(c)})|=2,$
a contradiction.
\end{proof}

We are now in a position to prove Theorem \ref{thm jordan ade}.

\begin{proof}[Proof of Theorem \ref{thm jordan ade}]
Assume that $D$ is not of $ADE$-type.  Then, by Lemma \ref{lem-d42}, there exists 
a subset $\caly\subset [\cala]$ of size $5$ such that if we let $Y:=\{\gt(x)\mid x\in\caly\},$
then (1)\, $H=\lan Y\ran$ is isomorphic to a central
quotient of $\Weyl_2(\tilde{D}_4)$.  (2)\, 
The diagram of $Y$ 
is the complete bipartite graph $K_{3,2}$ and (3)\, no $4$-subset
of $Y$ generates a subgroup isomorphic to $\Symm{4}$.

Let the parts of $Y$ be $\{a_1, a_2, a_3\}$ and $\{b_1, b_2\}$.
As we already noted, Lemma \ref{lem tatb2} and Lemma \ref{lem halfhalf}
imply that $\e-c\in [\cala],$ for all $c\in\caly$.

Hence, after perhaps interchanging $c$ with $\e-c$ for $c\in \{a_1, a_2, a_3\},$
using Corollary \ref{cor tatb3} and Lemma \ref{lem 3chalfa1-b}, we may assume that 
\[\textstyle{
N_{b_1,a_i}=3C(\half)_{b_1,a_i},\quad i=1,2,3}.
\] 
But then interchanging $b_2$ with $\e-b_2$ if necessary
we see that for at least two of $\{a_1, a_2, a_3\}$ say $a_1$ and $a_2$
we have 
\[\textstyle{
N_{b_2,a_i}=3C(\half)_{b_2,a_i},\quad i=1,2}.
\] 
But now taking $b_1,a_1,b_2,a_2$ in place of $a, b, c, d$ in Lemma \ref{lem nx}(2),
we get a contradiction. 
This completes the proof of the Theorem.  
\end{proof}

\begin{examples}
Finally we observe that, in a certain sense, the converse to 
Theorem \ref{thm jordan ade} holds. More precisely, for any field $\ff$
of characteristic not $2$ and for any $ADE$-type $X_n$,
there is a Jordan-axial $\ff$-algebra $A$ of Clifford type
such that 
\begin{enumerate}
{\it
\item
$A$ is a $3$-transposition algebra with respect to $\cala$;
\item
$D_{[\cala]}$ is a conjugacy of 
$3$-trans\-pos\-i\-tions of type $X_n$;
\item the Miyamoto 
group $G_{[\cala]}$ is isomorphic to one of the groups  $\Weyl(X_n)$ or $\Weyl(X_n)/Z(\Weyl(X_n))$.
(The possible groups being listed in Proposition \ref{prop-weyl}.)}
\end{enumerate}

Consider a root 
system $\Phi$ of type $X_n$.
Let $E$ be 
the Euclidean space containing (and spanned by) $\Phi$ and $E_{\zz}$ 
the root lattice in $E$, the $\zz$-span of $\Phi$. We assume that each 
root in $\Phi$ is of length $1$. Then the values of the inner product 
on $E_\zz$ belong to $\frac{1}{2}\zz$. 
(For instance, in the standard action of $\Weyl(A_m)=\Symm{m+1}$ on its 
permutation module $\rr^{m+1}$ equipped with the dot product,
the roots corresponding to transpositions
have  square length $2$ and inner-products $\pm 1$.)
Hence $V=E_\zz\otimes_\zz \ff$ is 
a vector space over $\ff$ of dimension $n$ endowed with a symmetric 
bilinear form $B$ such that $q(\bar r):=B(\bar r,\bar r)=1$ for 
all $r\in\Phi$. Here we use the notation $\bar e=e\otimes 1_{\ff}\in V$ 
for $e\in E_{\zz}$.

The Weyl group $W=\Weyl(X_n)$ of $\Phi$ generated by the reflections in all 
$r\in\Phi$ acts naturally on $\Phi$ and $E_{\zz}$ and hence on $V$. 
Namely, the reflection in a root $r$ acts on $V$ as the reflection 
in the corresponding vector $\bar r$. Let $\widehat W$ be the (isomorphic) 
image of $W$ in $GL(V)$. 

Consider $A=J(V,B)$ and
take $\cala=\{a=\frac{1}{2}(\e+\bar r)\mid r\in\Phi\}$. 
It follows from \S\ref{sect jordan} and the discussion above that $\cala$ is a set of 
$\half$-axes generating $A$.
The Miyamoto involution $\tau(a)$, for $\frac{1}{2}(\e+\bar r)=a\in \cala$,
fixes $\e\in A$ and acts as the negative of the reflection in $\bar{r}$ on $V$.
Therefore the group $G$ generated by 
the Miyamoto involutions for $\cala$ is a subgroup of index at 
most $2$ of the group $\lan-{\rm id}_{V}\ran\widehat W$.

The order of the product 
of two Miyamoto involutions is the same as the order of the product of the 
corresponding reflections. Hence $G$ is a group of $3$-transpositions
isomorphic to $\widehat W$ or $\widehat W/\lan-{\rm id}_{V}\ran$. The
second case occurs only if $\lan-{\rm id}_{V}\ran$ is in $\widehat W$ but not
in its subgroup generated by negative reflections. This in turn
happens if and only if $-{\rm id}_V \in \widehat W \sminus
\widehat W'$. The only such example is $\Weyl(E_7)$ with $\widehat W$
isomorphic to $\lan-{\rm id}_{V}\ran \times Sp_6(2)$ but $G$ isomorphic
to $Sp_6(2)$.

The space $(V,B)$ may have a nontrivial radical (depending
upon the type $X_n$ and the characteristic of $\ff$), in which case
there is a further example  $J(\tilde{V},\tilde{B})$
corresponding to $\tilde{V}=V/{\rm Rad}(V,B)$.
\end{examples}

\noindent
{\bf Acknowledgment.}  
We would like to thank  Felix Rehren for part of the proof
of Lemma \ref{lem ab=0}, and for various useful remarks.
We are also grateful to Holger Petersson for carefully going
through \S 5.


\end{document}